\documentclass{amsart}
\usepackage{amsmath,graphicx,verbatim, amsfonts,amssymb,amsthm,pictexwd,comment, amscd}
 \usepackage{url}
\usepackage{subfig, enumerate}
\usepackage{graphics}
\theoremstyle{plain}

\newcommand{\Z}{\mathbb{Z}}

\newcommand{\Q}{\mathbb{Q}}
\newcommand{\R}{\mathbb{R}}

\newcommand{\Int}{\mbox{Int}}
\newtheorem{theorem}{Theorem}[section]
\newtheorem{corollary}[theorem]{Corollary}
\newtheorem{lemma}[theorem]{Lemma}

\theoremstyle{definition}
\newtheorem{definition}[theorem]{Definition}
\newtheorem{remark}[theorem]{Remark}

\numberwithin{equation}{section}

\begin{document}
\title{The Triple Lattice PETs}
\author{Ren Yi}
\address{Department of Mathematics, Brown University, Providence, Rhode Island 02912}
\email{renyi@math.brown.edu}
\date{}
\begin{abstract}
Polytope exchange transformations (PETs) are higher dimensional generalizations of interval exchange transformations (IETs) which have been well-studied for more than 40 years. A general method of constructing PETs based on multigraphs was described by R. Schwartz in 2013. In this paper, we describe a one-parameter family of multigraph PETs called the triple lattice PETs. 

We show that there exists a renormalization scheme of the triple lattice PETs in the interval $(0,1)$. We analyze the limit set $\Lambda_\phi$ with respect to the parameter $\phi=\frac{-1+\sqrt 5}{2}$. By renormalization, we show that $\Lambda_\phi$ is the limit of embedded polygons in $\R^2$ and its Hausdorff dimension satisfies the inequality $1< \dim_H(\Lambda_\phi) = \log(\sqrt 2-1)/\log(\phi)<2$.
\end{abstract}

\maketitle

\section{Introduction}
Polytope exchange transformations (PETs) are dynamical systems that generalize interval exchange transformations (IETs). The definition of PETs is given as follows: 
\begin{definition}
Let $X$ be a polytope. A \textit{polytope exchange transformation (PET)} is determined by two partitions of small polytopes $\mathcal A = \{A_i\}_{i=1}^m$ and $\mathcal B = \{B_i\}_{i=1}^m$ of $X$.  For each $A_i \in \mathcal A$, there exists a vector $V_i$ satisfying the property that 
\[
B_i = A_i + V_i.
\]

A PET $f: X \to X$ is defined by the formula:
\[
f(x) = x+ V_i, \quad \forall x \in \Int(A_i).
\]
We call $V_i$ a translation vector of $f$ on $A_i$. Note that the PET $f$ is not defined on the boundary $\partial A_i$ for each $i$. 
\end{definition} 

Now, we introduce a coding of system of PETs. Given an initial point $p \in X$, we associate its \textit{coding} which is a sequence $\omega=\omega_0\omega_1\cdots$ for $\omega_n \in \{1, \cdots, m\}$ defined by 
\[
\omega_n = i \quad \mbox{if $f^n(p) \in A_i$}.
\]
\begin{definition}
\begin{enumerate}
\item A point is called $\textit{periodic}$ if $f^n(p)=p$ for $n \in \Z_+$. 
\item Suppose $p \in X$ is a periodic point of the map $f$. There exists a  maximal subset $\Delta_p$ containing $p$ such that the coding for each point in $\Delta_p$ is same as the one for $p$. We call the set $\Delta_p$ a \textit{periodic tile} of $p$. 
\item The union of all periodic tiles in $X$ is called the \textit{periodic pattern} which is denoted by $\Delta$.
\end{enumerate}
\end{definition}
However, it is unnecessary that every point in $X$ is periodic. We are interested in the case when $\Delta$ is dense.
\begin{definition}
When $\Delta$ is dense, we define the \textit{limit set}, denoted by $\Lambda$, as the set of points such that every neighborhood of the point in $\Lambda$ intersects infinitely many periodic tiles.  
\end{definition}
Note that the limit set $\Lambda$ contains all the points with well-defined and arbitrarily long orbits. We call such points  \textit{aperiodic points}. The union of all aperiodic points is called \textit{aperiodic set} with notation $\Lambda'$.

\subsection{Backgound}
The one-dimensional example of PETs are interval exchange transformations (IETs) which have been studied extensively (\cite{K1}, \cite{V}, \cite{Y} and \cite{Z}).  One simple construction of of dimension two or higher PETs comes from the products of IETs. Haller generalize this construction by introducing \textit{rectangle exchange transformations} \cite{H}. Haller establishes criterions of minimality for rectangle exchange transformations.

The dynamical systems of \textit{piecewise isometry} are closely related to PETs. Early examples of piecewise isometries  are studied in the paper \cite{AKT} and \cite{G}. Let $T: X \to X$ be a piecewise isometry on a polytope $X$. If we restrict the isometry on each piece to either a translation or a rotation by $q\pi$ for $q \in \Q$, we call the map $T$ a \textit{piecewise rational rotation}. (See \cite{AE}, \cite{GP}, \cite{LKV} and \cite{L1} for references.) There is a natural construction of PETs from piecewise rational rotations. Let $T: X \to X$ be a piecewise rational rotation. There exists a PET $\tilde T: \tilde X \to \tilde X$ in the covering space $(\tilde X, \pi)$ of $X$ so that $\pi \circ \tilde T = T \circ \pi$. 

In \cite{S4}, Schwartz introduce a general method of constructing PETs in all dimensions based on multigraphcs (see Section \ref{multi}). Let $G$ be a  multigraph such that each vertex is labeled by a convex polytope and each edge is labeled by a Euclidean lattice. Moreover,  a vertex is incident to an edge if and only if the vertex label is a fundamental domain of the edge label. There is a functorial homomorphism between the fundamental group $\pi_1(G, x)$ for a base vertex $x \in G$ and the group $\text{PET}(X)$ of PETs defined on the labeled polytope $X$ of $x$. We call the resulting systems \textit{the multigraph PETs}. Schwartz provides a one-parameter family of PETs called \textit{octagonal PETs} which corresponds to bigons (two vertices connected by two edges). He shows that there is  a local equivalence between outer billiards on semi-regular octagons and octagonal PETs. 

Few general results of PETs or piecewise isometries are known. Gutkin and Haydn \cite{GN} proved that piecewise isometries have zero topological entropy in dimension two. In the paper \cite{B}, Buzzi shows that the statement holds true in all dimensions. Our lack of understanding comes from the dynamics on the set whose orbits are not periodic. For example, we would like to understand the minimality or ergodicity on the set. More importantly,  we would like to know if it is possible to construct a recurrent PET in dimension two or higher. 

The scheme of renormalization is used to understand PETs (or more generally, piecewise isometries) in a lot of cases. Renormalization is a tool to zoom into the space and accelerate the orbits of points along time. To see this, we provide a basic definition: 
\begin{definition}
Let $Y$ be a subset of $X$. Given a map $f: X \to X$, the \textit{first return} $f|_Y: Y \to Y$ is a map assigns every point $p \in Y$ to the first point in the forward orbit of $p$ lying in $Y$, i.e.
\[
f|_Y(p) = f^n(p) \quad \mbox{where $n=\min\{f^k(p) \in Y\}  \quad k > 0$}.
\]
\end{definition}

For IETs, Rauzy induction \cite{R} provides a renormalization scheme. A classical example of renormalizable piecewise isometries is described in the survey paper \cite{G1} by Goetz. In the paper \cite{L1}, Lowenstein develops a general theory of piecewise rational rotations. Hooper gives the first example of PETs in 2-dimensional parameter space which is invariant under renormalization \cite{HO}. The renormalization scheme arises from collapsing the reducing the loops in Truchet tilings. A recent paper \cite{BB} studies an example of piecewise isometries which is very similar to Hooper's example. A renormalization scheme of the system is discovered through a symbolic coding of the system. In \cite{S4}, Schwartz shows that there is a renormalization scheme on the one-parameter family of octagonal PETs. Moreover, he finds that the hyperbolic triangular group $(2,4,\infty)$ acts on the parameter space by linear fractional transformation as a renormalization symmetry group.

\subsection{Multigraph PETs}\label{multi}
Here we describe a construction of PETs in all dimensions called \textit{multigraph PETs}, which is introduced by Schwartz in \cite{S3}.  

\begin{definition}
A \textit{multigraph} is a monogon-free graph such that two vertices may be connected by more than one edge.
\end{definition}

Let $G=(V,E)$ be a multigraph such that every vertex $v \in V$ is labeled by a convex polytope and every edge $e \in E$ is labeled by a Euclidean lattice. Moreover, two vertex $v_0$ and $v_1$  are connected by an edge $e$ iff the labeled convex polytopes $X$ and $Y$ for $v_0$ and $v_1$ respectively are fundamental domains of the  lattice $L$ associated to the edge $e$.  Then, we can translate almost every point in $X$ to $Y$ by a unique vector in the lattice $L$. Therefore, each closed path based on a vertex $v_0$ in $G$ corresponds to a PET on $X$.  The resulting system is called the multigraph PETs. The construction is considered functorial: there is a functor from the the fundamental group $\pi_1(G, x)$ based on a vertex $x$ to the group of PETs defined on the labeled convex polytope $X$ of $x$, i.e.
\[
\pi_1(G,x) \to \text{PET}(X).
\]

\subsection{Triple Lattice PETs}
In this section, we describe the construction of triple lattice PETs whose corresponding multigraph are triangles. Let $F_0$ be a parallelogram determined by the vectors $(2,0)$ and $(s, \sqrt 3 s)$ for $s \in [0,1)$. Let $r_0, r_1$ and $r_2$ be the reflections about the lines in the directions of cubic roots of unity passing through the origin. Let $D_6$ be the dihedral group of order $6$ generated by the reflections $r_0, r_1$ and $r_2$. The 6 parallelograms in Figure \ref{fig:triple_lattice} are in the orbit of $F_0$ under the dihedral group $D_6$. Let $F_i$ be the parallelogram centered at the origin and translation equivalent to the one labeled as $F_i$  in Figure \ref{fig:triple_lattice} for $i\in \{0,1,2\}$. Let $L_i$ be the lattice generated by the vectors which are the sides of the parallelogram labeled by $L_i$ in Figure $\ref{fig:triple_lattice}$ for $i \in \{0,1,2\}$. 
\begin{figure}[h]
\centering
\includegraphics[width=1.6in]{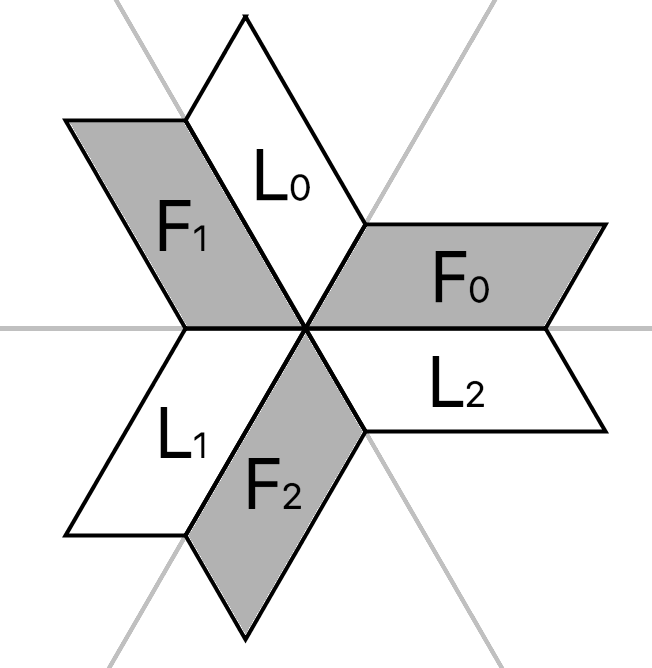}
\hspace{4em}
\includegraphics[width=1.7in]{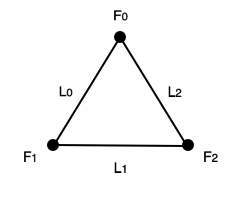}
\caption{The scheme for the triple lattice PETs}
\label{fig:triple_lattice}
\end{figure}

We will verify the fact that $F_i$ and $F_{(i+1) \mod 3}$ are fundamental domains of the lattice $L_i$ in the later section. It follows that for almost every point $p\in F_i$, there exists a unique vector $V_p \in L_i$ such that $p+V_p \in F_{(i+1) \mod 3}.$ To define the triple lattice PETs,  we consider the set $X_s' = \displaystyle\bigcup_{i=0}^2 F_i$. Given $p \in F_i$, define the map $g_s : X'_s \to X'_s$ by the formula:  
\[
g_s(p) = p + V_p \in F_{(i+1) \mod 3}, \quad V_p \in L_i. 
\]
We leave $g$ undefined on the point $p$ when $V_p$ is not uniquely determined by $p$.
\begin{definition}
Let $X_s = F_0$. The \textit{triple lattice PET} is given by the map $f'_s: X_s \to X_s$ such that
\[
f'_s = (g_s)^3.
\] 
We denote the system by $(X_s,f'_s)$. 

In order to obtain symmetries, we introduce a \textit{cut-and-paste} operation (Section \ref{cut_and_paste}) which can be seen as a simple rearrangement of the map $f'_s$ (Figure \ref{fig:cut_paste}). We call the the resulting system after cut-and-paste the triple lattice PETs denoted by $(X_s, f_s)$. 
\end{definition}

\begin{figure}[h]
\centering
\includegraphics[scale=0.56]{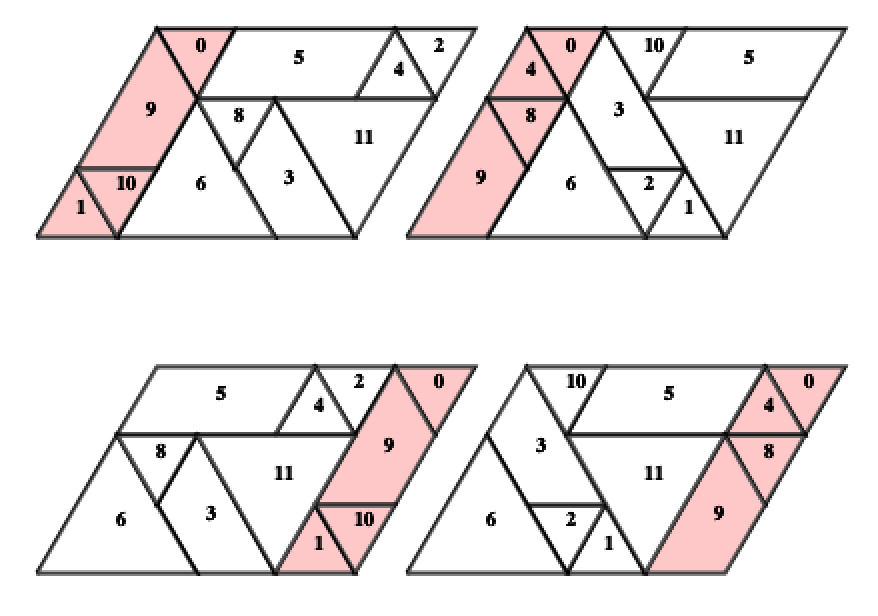}
\caption{An illustration of the cut-and-paste operation on $X_s$ for $s=3/4$. The top figures show the partitions $\mathcal A'$ and $\mathcal B'$ for the triple lattice PET $f'_s$. The bottom figure show the modified partitions $\mathcal A$ and $\mathcal B$ for the map $f_s$ obtained by translating all the elements of $\mathcal A'$ and $\mathcal B'$ in the red parallelograms to the right and then move the resulting parallelogram centered at the origin again.}
\label{fig:cut_paste}
\end{figure}

\subsection{Main Results}
Let $R_1, R_2: (0,1) \to (0,1)$ be the maps given by
\[
R_1(s) =  \frac{s}{1+ s - s\lfloor 1/s \rfloor} \quad \mbox{and} \quad R_2(s) = \frac{1-s}{s}.
\]
We define the renormalization map $R: (0,1) \to (0,1)$ as follows
\[
R(s)=\left\{
\begin{array}{l l}
R_1(s) & \mbox{if $s \in (0,\frac{1}{2})$, }\\
R_2(s) & \mbox{if $s \in [1/2, 2/3)$, }\\
R_1  R_2(s) & \mbox{if $s \in $[2/3, 1)}.
\end{array}
\right.
\]
\begin{theorem}[Renormalization]\label{main}
Let $s \in (0,1)$ and $t=R(s)$. There exists a subset $Y_s$ such that the first return map $f_s|_{Y_s}: Y_s \to Y_s$ satisfies
\[
f_s |_{Y_s} = \psi_s^{-1} \circ f_t \circ \psi_s.
\]
where $\psi_s$ is the map of similarity such that $Y_s \mapsto X_t$. 
\end{theorem}

\begin{figure}[h!]
\centering
\includegraphics[scale=0.54]{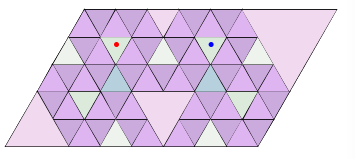}
\includegraphics[scale=0.56]{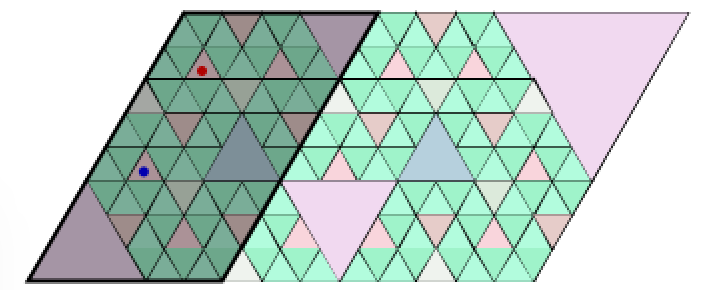}
\caption{The top figure is an illustration of $f_t$ for $t=5/8=R(8/13)$. Given a point (red) in $X_t$, its image is shown blue. The figure on bottom shows $Y_s$ (lightly shaded) for $s=8/13$. The image of the given point (red) under the first return $f_s|_{Y_s}$ is shown in blue. The first return map $f_s|_{Y_s}$ is conjugate to $f_t$ by a similarity.}
\vspace{3.5em}
\includegraphics[scale=0.6]{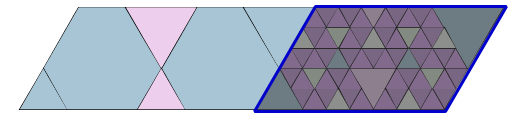}
\caption{$Y_s$ lightly shaded for $s=5/18$ and $R(5/18)=5/8$}
\vspace{3.5em}
\includegraphics[scale=0.6]{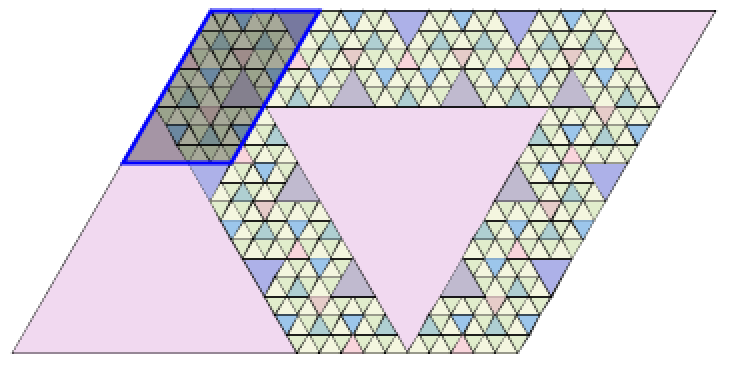}
\caption{$Y_s$ lightly shaded for $s=18/23$ and $R(18/23)=5/8$}
\end{figure}

The renormalization theorem allows us to deduce corollaries on the limit set $\Lambda_\phi$ for $\phi=\frac{\sqrt 5-1}{2}$. Note that $\phi$ is the only fixed point under the renormalization map $R$.  The limit set $\Lambda_\phi$  can be obtained by a sequence of substitutions on the isosceles trapezoids $A, B$ and  $C$ in $X_\phi$ as shown Figure 6-9. 

\begin{theorem}
The limit set $\Lambda_\phi$ is the limit of embedded polygons in $\R^2$.
\end{theorem}

The substitutions generate an iterated function system of similarities, and the limit set $\Lambda_\phi$ is a self-similar set. We compute the Hausdorff dimension of $\Lambda_\phi$.
\begin{theorem} 
The Hausdorff dimension of the limit set $\Lambda_\phi$ satisfies the property:
\[
\dim_H(\Lambda_\phi) = \frac{\log(-1+\sqrt 2)}{\log \phi}=1.83147\dots<2.
\]
\end{theorem}

\begin{corollary}
The limit set $\Lambda_\phi$ has Lebesgue measure zero. 
\end{corollary}

 \begin{figure}[h]
 \centering
 \includegraphics[scale=0.47]{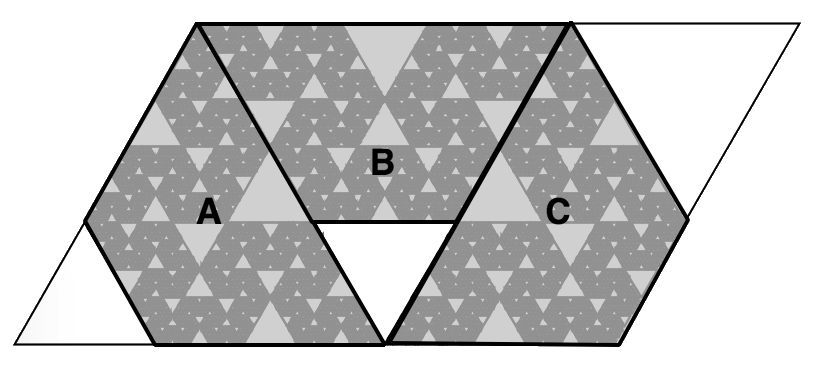}
 \caption{Isosceles trapezoids $A, B, C$ in $X_\phi$}
 \vspace{2em}
\centering
\includegraphics[scale=0.4]{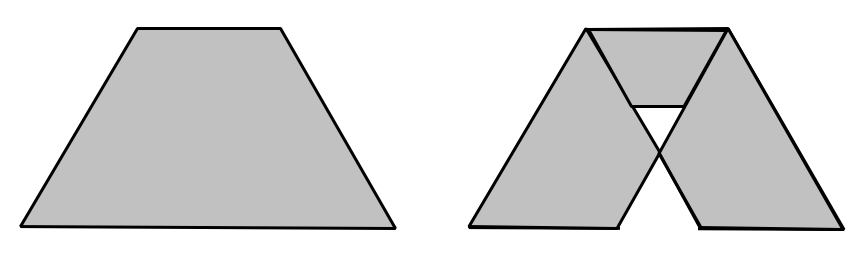}
\caption{Substitution rule of the isosceles trapezoids}
\vspace{2em}
\includegraphics[scale=0.45]{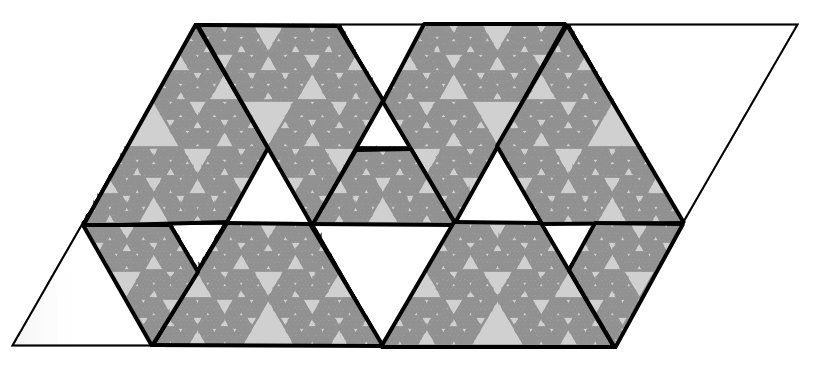}
\includegraphics[scale=0.45]{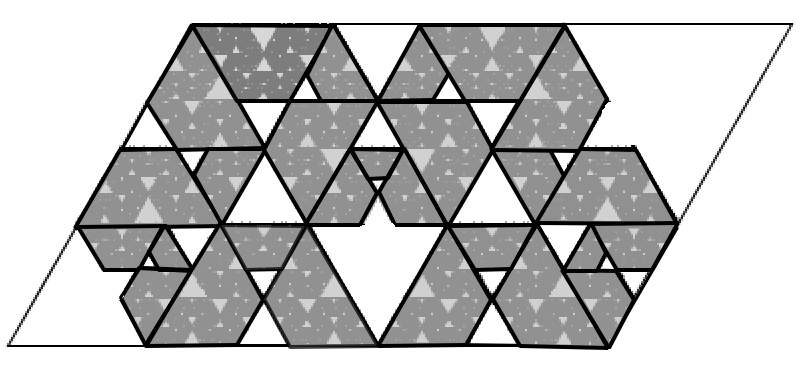}
\caption{The iterations of the substitution rule on trapezoids $A, B$ and $C$}
\vspace{2em}
\includegraphics[scale=0.15]{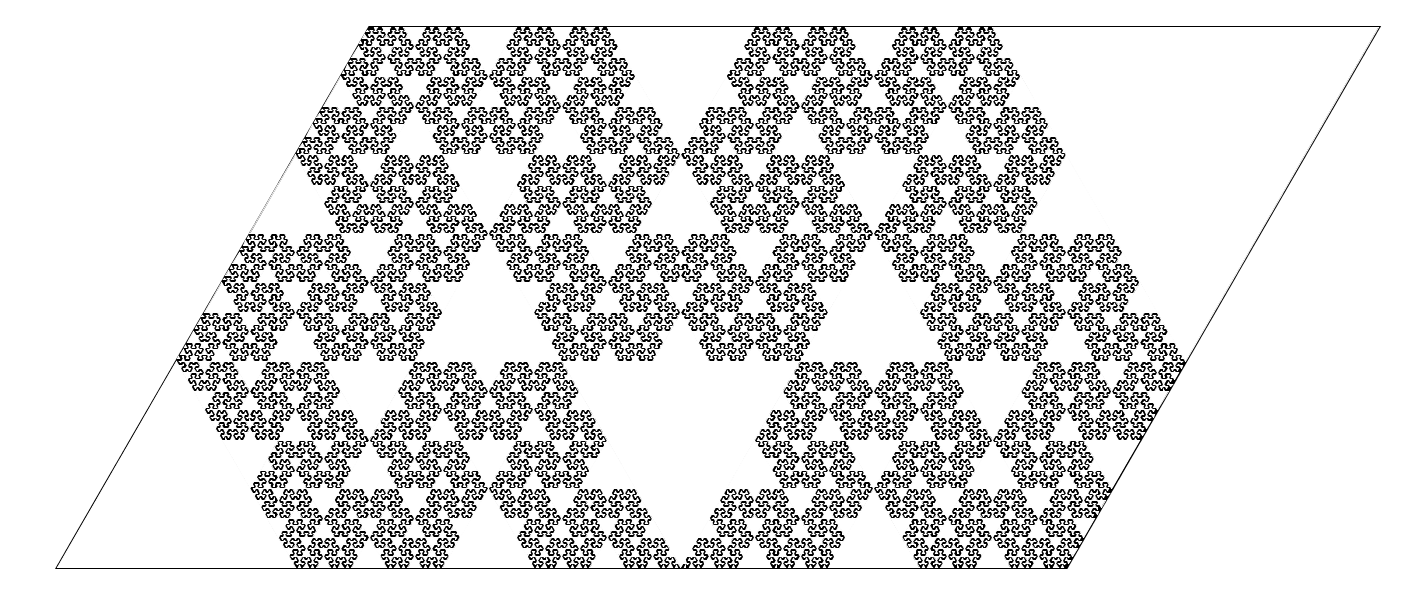}
\caption{The limit set $\Lambda_\phi$}
\end{figure}

\subsection{Outline}
In Section 2, we provide basic definitions and properties related to the triple lattice PETs. Properties of the renormalization map $R$ are discussed in this section. 

Section 3 introduces the fiber bundle method which is a key tool to prove the main theorem. 

In Section 4, we prove the Renormalization Theorem using two inductions on the parameter space. 

Symmetries of the triple lattice map $f_s$ are discussed in Section 5. 

In Section 6, we studies the limit set $\Lambda_\phi$ following a description of the substitution rule on the isosceles trapezoids. We use the substitution to compute the Hausdorff dimension and Lebesgue measure of the limit set $\Lambda_\phi$. 

All computational data are provided in Section 7. 

\subsection{Acknowledgement} The author would like to thank her advisor Professor Richard Schwartz for his constant support, encouragement and guidance throughout this project. The author would also like to thank Patrick Hooper and Yuhan Wang for helpful suggestions at different steps of the paper. 

\clearpage
\newpage
\section{Preliminaries}
\subsection{Fundamental Domains}
We check the fact that each parallelogram $F_i$ is a fundamental domain of $L_{i}$ and $L_{(i-1) \mod 3}$ for $i\in \{0,1,2\}$. According to \cite{S4}, it is sufficient to check two facts: 
\begin{enumerate}
\item The parallelogram $F_i$ and the lattice quotient $\R^2/L_i$ or $\R^2/ L_{(i-1)\mod 3}$ have the same volume. It is easy to see that 
\[
Area(\R^2 / L_j ) = Area(F_i) = 2 \sqrt 3 s.
\]
for $i \in \{0,1,2\}$.
\item The union 
\[
\bigcup_{V \in L} (F_i+V) 
\]
provides a tiling of $\R^n$ where $L=L_i$ or $L=L_{(i-1)\mod 3}$ for $i\in \{0,1,2\}$.

Here we only discuss the case when $i=0$. The same argument can be applied similarly by reflections. Recall that the lattice $L_0$ is generated by the vectors 
\[
V_0=s(1, \sqrt 3) \quad \mbox{and} \quad  V_1 = (-1,\sqrt 3).
\]
and $F_0$ is a parallelogram determined by the vectors $s(1, \sqrt 3)$ and $(2,0)$. First we notice that $F_0$ and $F_0 +V_0$ form adjacent tiles meeting at the top side of $F_0$.  Consider a band $\mathcal S$ between lines $l_1:y= \sqrt 3 x$ and $l_2: y=\sqrt 3 (x-2)$. The union $\displaystyle\bigcup_{m\in \Z}(F_0+ mV_0)$ form a tiling of $\mathcal S$ in $\R^2$. Notice that by applying $V_1$, the band $\mathcal S$ translates to the adjacent band $\mathcal S - (2,0)$. Therefore, the union $\displaystyle\bigcup_{m,n\in \Z}(F_0+ mV_0+nV_1)$ is a tiling in $\R^2$. 

Similarly, the lattice $L_2$ is generated by the vectors $V_0=(2,0)$ and $V_1=s(1, -\sqrt 3)$. The union $\displaystyle\bigcup_{m\in \Z} F_0+mV_0$ form a tiling in the band $\mathcal S$ between the horizontal axis and $y=\sqrt 3 s$. Moreover, The vector $V_1$ translates $\mathcal S$ to $\mathcal S+(0,\sqrt 3)s $. Consequently,  the union $\displaystyle\bigcup_{m,n\in \Z}(F_0+ mV_0+nV_1)$ is a tiling in $\R^2$. 
\end{enumerate}

\subsection{Cut-and-Paste Operation}\label{cut_and_paste}
In this section, we explain the cut-and-paste operation in detail. Suppose that $s \in (0,1)$. Let $a_1$ be the first non-zero integer of the continued fraction expansion of $s$, i.e.
\[
a_1 = \lfloor 1/s \rfloor.
\] 
Let $\mathcal A', \mathcal B'$ be the partitions of $X_s$ defining a triple lattice PET $f'_s$. Let $Y_s \subset X_s$ be a parallelogram (shown as red in the top figure) such that $Y_s$ and $X_s$ share the same lower left vertex $v$. The sides of $Y_s$ are determined by the vectors 
\[
v+(s, \sqrt 3s) \quad \mbox{and} \quad v+(2-2a_1s,0), 
\]
The cut-and-paste operation $\mathcal{OP}$ on the parallelogram $X_s$ can be described as follows:  for each point $p \in X_s$, if $p \in Y_s$, then we translate $p$ by the vector $(2-2a_1s,0)$. Otherwise, translate $p$ by  $(-2a_1s, 0)$.  Note that $X_s$ remains the same after the operation $\mathcal{OP}$. We obtain the modified partitions $\mathcal A$ and  $\mathcal B$ on $X_s$ which produce a new family of PETs $f_s: X_s \to X_s$. 
\begin{figure}[h]
\centering
\includegraphics[scale=0.45]{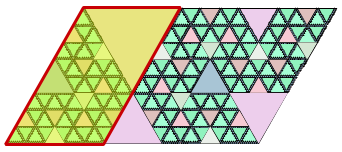}
\includegraphics[scale=0.45]{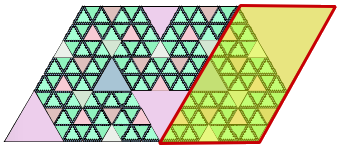}
\caption{The figure on the left shows the periodic pattern $\Delta'_s$ of the map $f'_s$ before the cut-and-paste operation $\mathcal{OP}$ and the figure on the right shows the $\Delta_s$ of $f_s$ after $\mathcal{OP}$.}
\label{fig:cut_paste_pattern}
\end{figure}

\subsection{Algorithm for generating Periodic Tiles} 
Let $\mathcal A_0=\{A_i\}_{i=1}^m$ and $\mathcal B_0=\{B_i\}_{i=1}^m$ be the partitions of polygons such that the map $f_s$ is determined by $\mathcal A_0$ and $\mathcal B_0$. For integers $n \geq 1$, we inductively define $\mathcal A_n$ to be the collection of polygons
\[
f^{-n}(f^n(P) \cap A_i), \quad \mbox{$P \in \mathcal A_{n-1}$ and $A_i \in \mathcal A_0$}.
\]
The partition $\mathcal A_n$ is a refinement of $\mathcal A_{n-1}$. For each $n \geq 2$, the iteration $f^n$ is not defined on 
\[
\bigcup_{P \in \mathcal A_n} \partial P.
\]
Let $P \in \mathcal A_n$ be an open polygon. Every point in $P$ must have the same codings. It follows that if a point $p \in P$ is periodic of period $n$, then all points of $P$ are periodic with period $n$. Remark that each periodic tile is convex because of the fact that intersections of convex polygons are convex. 

\subsection{Renormalization Map}
We discuss the properties of the renormalization map $R: (0,1) \to (0,1)$ defined in section \ref{main} through continued fraction expansions. Suppose that $s \in (0,1)$ is rational and has continued fraction expansion (c.f.e) 
\[
(0; a_1, a_2, \cdots, a_n).
\] 
If $s \in (0,\frac{1}{2})$, then $a_1 \geq 2$ in the c.f.e of $s$. The renormalization map $R$ has the action 
\[
(0; a_1, a_2, \cdots, a_n) \mapsto (0; 1, a_2, \cdots, a_n).
\]
For example, $5/23=(0;4,1,1,2)$ and $R(5/23)=\frac{5}{8}=(0;1,1,1,2)$.

If $s \in [1/2, 2/3)$, then its c.f.e. can be written in the form of $(0;1,1,\cdots, a_n)$. The renormalization map $R$ can be viewed as a shift on the fractional part of the c.f.e., i.e.
\[
(0; 1, 1, a_3, \cdots, a_n) \mapsto (0; 1, a_3, \cdots, a_n).
\]
The parameter $\phi=\frac{-1+\sqrt 5}{2}=(0; 1,1,\cdots)$ is the only fixed point under the renormalization map $R$.

If $s \in [2/3, 1)$, then $a_1=1$ and $a_2 \geq 2$ in the c.f.e.. The map $R$ first shift the fractional part of the c.f.e. to the right, and then reduce $a_2$ to $1$, i.e.
\[
(0; 1, a_2, \cdots, a_n) \mapsto (0; 1, a_3, \cdots, a_n).
\]
Therefore, the renormalization map $R$ has the action 
\[
(\frac{n-1}{n}, \frac{n}{n+1}) \mapsto (\frac{1}{2},1) \quad \mbox{and} \quad (\frac{1}{n}, \frac{1}{n-1}) \mapsto (\frac{1}{2},1). 
\]
Moreover, when $s$ is irrational, all arguments hold true by taking $n \to \infty$. 
\begin{figure}[h]
\centering
\includegraphics[scale=0.5]{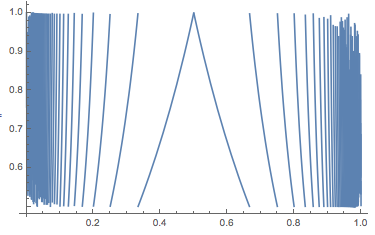}
\caption{The plot of renormalization map $R$ with range [$\frac{1}{2},1$]}
\end{figure}

\subsection{Hausdorff dimension}
In this section, we provide basic definitions related to Hausdorff dimension and a formula to calculate it (see \cite{FA} for reference). 
\begin{definition}
Suppose that $F$ is a compact set of $\R^n$ and $s\geq 0$. Define
\[
H^s_\delta(F)=\inf\{\sum^\infty_{i=1} \mbox{diam}(U_i)^s| \mbox{diam}(U_i) \leq \delta \quad \mbox{and} \quad F \subset \bigcup_{i=1}^\infty U_i\}.
\]
The \textit{Hausdorff measure} is defined to be the limit
\[
H^s(F) = \lim_{\delta\to 0} H^s_\delta(F).
\]
\end{definition}
There is a critical value $s_0$ of $s$ at which $H^s$ jumps from $\infty$ to $0$. The value $s_0$ is called \textit{Hausdorff dimension}. Formally, the Hausdorff dimension $\dim_H F$ of a compact set $F$ is defined as follows:
\begin{definition}
\[
\dim_H F = \inf\{s\geq 0: H^s(F)=0\}=\sup\{s: H^s(F)=\infty\}.
\]
If $s=\dim_H F$, then $0 < H^s(F) <\infty$. 
\end{definition}

Here we state a classical result in \cite{FA} to compute Hausdorff dimensions of self-similarity sets. Let $\{\sigma_0, \sigma_1, \cdots, \sigma_m \}$ be a set of similarities with scaling ratios $0<c_i<1$ for $1\leq i \leq m$. If the following conditions are satisfied:
\begin{enumerate}
\item attractor condition: the set $F$ is an attractor i.e.
\[
F=\bigcup_{i=1}^m \sigma_i(F),
\]
\item open set condition: there exists non-empty bounded open set $V$ such that 
\[
V \supset \bigsqcup_{i=1}^m\sigma_i(V),
\]
\end{enumerate}
then the Hausdorff dimension $\dim_H F=s$, where $s$ is given by
\[
\sum_{i=1}^m c_i^s=1.
\]

\subsection{Computer Assistance}
We give a proof for the main renormalization theorem and symmetrical properties of periodic patterns (see Section 3 and 4) for the map $f_s$ with computer assistance. The proof involves calculations to determine if a given pair of polyhedra are nested or disjoint. We scale all the convex polyhedra so that all the calculations are done in integers or half integers. Hence, there is no roundoff error. The pictures of partitions, periodic patterns and limit sets are taken from my java program. The program also do all the calculations. The program can be downloaded from the URL
\[
https://www.math.brown.edu/\sim renyi/triple\_ lattice/project.html
\]

\section{The Fiber Bundle Picture}\label{fiber}
Suppose $s \in (0,1)$ and $t=R(s)$. We want to show the first return map $f_s|_{Y_s}$ conjugates to the map $f_t$ by a map of similarity $\psi_s$. Since it is impossible to apply calculations on every $s \in (0,1)$, the idea of the proof is to reduce all the calculations to finitely many computations on the fiber bundle $\mathcal X$ which is a convex polyhedron in  $\R^3$. 

The construction is inherited from \cite{S3} in Chapter 26. Define $\mathcal X \in \R^3$ as follows
\[
\mathcal X = \{(x,y,s)| (x,y) \in X_s, s\in [\frac{1}{2}, 1]\}.
\]
The set $\mathcal X$ is a convex polyhedron and a fiber bundle over $[1/2, 1)$ such that the fiber above $s$ is the parallelogram $X_s$. Let $F: \mathcal X \to \mathcal X$ be the fiber bundle map given by the formula:
\[
F(x,y,s) = (f_s(x,y), s).
\]
It is easy to see that $F$ is a piecewise affine map. It is because for each $(x,y) \in X_s$, the map $f_s$ is of the format  
\[
f_s(x,y) = (x,y) + (m_0 s + m_1, (n_0 s + n_1)\sqrt 3)
\]
where $m_i, n_i \in \Z$ for $i=0,1$. If we vary the point $(x,y)$ and the parameter $s$ in a small neighborhood, the integers $m_i, n_i$ for $i=1,2$ will not change. 

\begin{definition}
A \textit{maximal domain} of $\mathcal X$ is a maximal subset in $\mathcal X$ where the fiber bundle map $F$ is entirely defined and continuous. 
\end{definition}

In other words, a maximal domain $D$ of $\mathcal X$ is a maximal subset of $\mathcal X$ such that for all $(x,y,s) \in \Int(D)$, the fiber bundle map $F: \mathcal X \to \mathcal X$ is in the form of
\[
F(x,y,s) = (x, y, s)  + \big( m_0 s + m_1, y+ (n_0s+n_1)\sqrt 3, 0\big)
\]
for the fixed integers $m_i, n_i$ where $i=0,1$. It follows that for each maximal domain $D$, there is a 4-tuple $(m_0, m_1, n_0, n_1)$ encoded the information of its translation vector. We call it the \textit{coefficient tuple} of the maximal domain $D$.  

The fiber bundle $\mathcal X$ is partitioned into 12 maximal domains $D_i$ for $i=0,\cdots, 11$. The vertices of the maximal domains are of the form 
\[
(\frac{1}{2}\frac{a}{q}, \frac{\sqrt 3}{2} \frac{b}{q}, \frac{p}{q}), \quad \mbox{for $a,b,p,q \in \Z$}
\]
and 
\[
(p,q) \in \{(1,2), (2,3), (1,1)\}.
\] 
The list of maximal domains along with their coefficient tuple is provided in Section \ref{data}.  

\begin{remark}
\begin{enumerate}
\item All maximal domains except for  $D_7$ have at least one vertices with $z$-coordinate greater than 2/3. Formally, we consider the subset $\mathbb E =\{(x,y,z): z\geq 2/3\}$ of $\R^3$. For $i=0,\cdots, 12$, the Lebesgues measure $\mu(D_i \cap \mathbb E)$ equals to $0$ when $i=7$ and strictly greater than $0$ otherwise. We say  $D_7$ \textit{degenerates} on the interval [2/3,1). 
\item The union of cross sections obtained by intersecting the maximal domains and the plane $z=s$ gives a partition $\mathcal A$ on $X_s$ which determines the triple lattice map $f_s$. 
\end{enumerate}
\end{remark}

\section{Proof of the Renormalization Theorem}
\subsection{Outline of the Proof}
The proof of the renormalization theorem relies on two inductions. We provide an outline here.  
\begin{itemize}
\item[--] We frist show that the main theorem holds true when $s\in [\frac{1}{2},\frac{2}{3})$ by applying calculations on 12 maximal domains in the fiber bundle $\mathcal X$. 
\item[--] By the similar technique, we check the renormalization theorem on the two other base cases when $s \in [2/3, 3/4)$ and $s \in [\frac{3}{4}, \frac{4}{5})$.

\item[--] We apply Induction I on the intervals of the format $[\frac{n-1}{n}, \frac{n}{n+1})$ for $n\geq 2$ to show that Theorem \ref{main} holds true for all $s \in [\frac{1}{2},1)$. More precisely, we find that for $s \in [\frac{n-1}{n}, \frac{n}{n+1})$, the first return map  $f_s|_{Y_s}$ is similar to $f_t|_{Y_t}$ for some $t \in [\frac{n}{n+1}, \frac{n+1}{n+2})$. 
Therefore, we have renormalization scheme proved on the intervals
\[
[\frac{3}{4}, \frac{4}{5}) \mapsto [\frac{4}{5}, \frac{5}{6}) \mapsto [\frac{5}{6}, \frac{6}{7}) \mapsto \cdots. 
\]
By taking the union
\[
\bigcup_{n\geq 2} [\frac{n-1}{n}, \frac{n}{n+1})=[\frac{1}{2},1),
\]
Theorem \ref{main} can be shown for $s \in [\frac{1}{2}, 1)$.

\item[--]  We apply induction II to verify Theorem \ref{main} when the parameter $s \in [\frac{1}{n}, \frac{1}{n-1})$ for all $n\geq 2$. We show that the periodic pattern for $s \in [\frac{1}{n}, \frac{1}{n-1})$ and $t \in [\frac{1}{n+1}, \frac{1}{n})$ are same up to scaling and adding on one central parallelogram. By this observation, we have renormalization scheme proved on the intervals
\[
[\frac{1}{2}, 1) \mapsto [\frac{1}{3}, \frac{1}{2}) \mapsto [\frac{1}{4}, \frac{1}{3}) \mapsto \cdots. 
\]
Consequently, the renormalization theorem is shown for $s\in (0,1)$ that is the union
\[
\bigcup_{n\geq 2}[\frac{1}{n}, \frac{1}{n-1})=(0,1).
\]
\end{itemize}

\subsection{Base Case 1.}\label{base}
Suppose $s \in I_0 =[1/2, 2/3)$ and $t=R(s) \in (\frac{1}{2}, 1]$. Now we define $Y_s$ (Figure \ref{base1}) which is the set for the first return map. Let $Y_s \subset X_s$ be the parallelogram determined by the vectors
\[
V_0+(s, \sqrt 3 s), \quad \mbox{and} \quad V_0+(2-2s,0). 
\]
where $V_0=(-\frac{s}{2}-1,-\frac{\sqrt 3s}{2}) \quad$ is the lower left vertex of $X_s$. 
\vspace{1em}
\begin{figure}[h]
\centering
\includegraphics[scale=0.55]{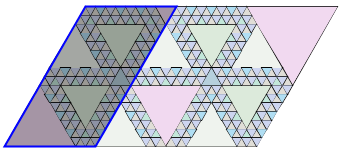}
\includegraphics[scale=0.55]{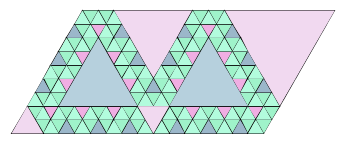}
\caption{$Y_s$ (lightly shaded) for $s=\frac{16}{25}$ and $\Delta_t$ for $t=R(s)=\frac{9}{16}$}
\label{base1}
\end{figure}

Before giving the explicit formula, we give an informal description of the similarity $\psi_s$  in Theorem \ref{main} mapping $Y_s$ to $X_t$. The map $\phi_s$ first translates the set $Y_s$ to center at the origin. Then, rotate the obtained parallelogram $Y_s$ counterclockwisely by $2\pi/3$ around the origin. Flip the obtained shape about the horizontal axis. Finally, scale the resulting parallelogram by $1/s$. Hence, the formula of $\psi_s: Y_s \to X_t$ for $s\in I_0$ is given as follows:
\[
\begin{bmatrix}
x \\
y
\end{bmatrix} \mapsto \frac{1}{s}
\begin{bmatrix}
-1/2 & -\sqrt 3/2 \\
-\sqrt 3 / 2 & 1/2
\end{bmatrix}
\begin{bmatrix}
x+s\\
y
\end{bmatrix}.
\]
Define the polyhedron $\mathcal Y \subset \mathcal X$ as follows
\[
\mathcal Y = \{(x,y,s)| (x,y \in Y_s), s \in I_0\}.
\]
The polyhedron $\mathcal Y$ is a fiber bundle over $I_0$ such that the fiber above $s$ is the parallelogram $Y_s$. We define the affine map $\Psi: \mathcal Y \to \mathcal X$ by piecing together all the similarities $\psi_s$ for $s \in I_0$
\[
\Psi(x,y,s) = \Big (\psi_s(x,y),R(s)\Big ).
\]
Note that the map $R$ has the action that 
\[
I_0=[1/2, 2/3) \mapsto (1/2, 1].
\]
The inverse map $\Psi^{-1}$ is given by the formula:
\[
\Psi^{-1}(x,y) = \Big( \psi^{-1}_s(x,y), \frac{1}{1+s}\Big ).
\]

Here we give the main calculation steps of the proof. For each maximal domain $D_i$ in $\mathcal X$, we check the following properties by computer:

\begin{enumerate}
\item There exists an integer $n > 0$ such that 
\[
F^n \circ \Psi^{-1}(D_i) \subseteq \mathcal Y 
\]
and 
\[
 F^{k} \circ \Psi^{-1}(D_i) \cap \mathcal Y = \emptyset, \quad \mbox{for all $0 \leq k < n$}.
\]
It means that the first return map on $\mathcal Y$ is well-defined and the fiber bundle map $F$ returns to the polyhedron $\mathcal Y$ as $F^n$. 
 
\item $F^n \circ \Psi^{-1} (D_i) \subseteq  \Psi^{-1} \circ F(D_i)$
\item Pairwise disjointness: set $M_i = F^n \circ \Psi^{-1}(D_i)$. 
\[
\quad \Int(M_i) \cap \Int(M_j) = \emptyset \quad \mbox{for $i \neq j$}.
\]
\item Filling: 
\[
\sum_{i=0}^{11} \text{Volume}(M_i)  = \text{Volume}(\mathcal Y).
\]
\end{enumerate}
The above computation shows that the first return map $F|_{\mathcal Y}$ on $\mathcal Y$ satisfies the equation 
\[
F|_{\mathcal Y} = \Psi^{-1} \circ F \circ \Psi.
\]
Therefore, the renormalization theorem on the first base case $I_0=[1/2, 2/3)$ has been proved.
\qed\\

\subsection{Base case 2 and 3} Consider $I_1=[2/3, 3/4)$ and $I_2=[3/4, 4/5)$. Suppose $s \in [2/3, 3/4)$ or $[3/4, 4/5)$. In these cases, we want to apply the similar proof as the one in the previous section. One difference is that the subset $Y_s$ is chosen differently than before. We define the subset $Y_s \subset X_s$ for $s\in I_j$ for $i=1,2$ as follows. Suppose $s \in I_j$ for $j \in \{1,2\}$. Let $V_0=(\frac{s}{2}-1, \frac{\sqrt 3s}{2})$ be the top left vertex of $X_s$. Consider a parallelogram $Y_s$ determined by the vectors 
\[
 V_0+ (2a_s, 0) \quad \mbox{and} \quad V_0 - b_s(1, -\sqrt 3) 
\]
where 
\[
a_s= 1- s \quad \mbox{and}  \quad b_s = 1 - \lfloor \frac{s}{a_s} \rfloor a_s.
\]
\begin{figure}[h]
\centering
\includegraphics[scale=0.45]{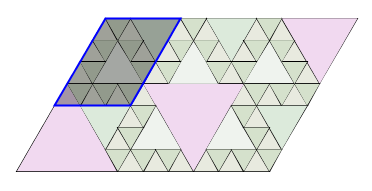}
\includegraphics[scale=0.45]{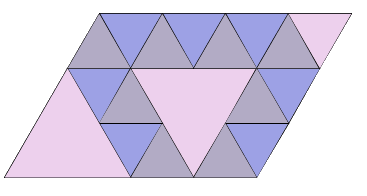}
\caption{$Y_s$ (lightly shaded) for $s=7/10$ and $\Delta_t$ for $t=R(s)=3/4$}
\label{base2}
\end{figure}

Here we show that the renormalization theorem holds true for $s\in I_1$ and $s \in I_2$.  
\begin{lemma}
Suppose $s \in I_j$ and $t = R(s)$ for $j\in \{1,2\}$. The first return map $f_s|_{Y_s}$ on $Y_s$ conjugates to $f_t$ by a map of similarity $\psi_s$. 
\end{lemma}
We want to use the argument as the one in the previous section. Similarly, we define the polyhedron $\mathcal Y(I_j)$ as a subset of $\mathcal X$ over the interval $I_j$, i.e.
\[
\mathcal Y(I_j) = \{(x,y,s)| (x,y) \in Y_s, s\in I_j\} \quad \mbox{for} \quad j=1,2.
\]
Suppose $s \in I_j$ and $t=R(s) \in (1/2, 1]$. The map of similarity $\psi_s: Y_s \to X_t$ is defined similarly as the in the Section 3.1 but with a different scaling factor $1/b_s$ for $b_s=(j+1)s-j$. Let $s\in I_j$ for $j\in \{1, 2\}$, the formula of $\psi_s: Y_s \to X_t$ is given by the following formula: 
\[
\begin{bmatrix}
x \\
y \\ 
\end{bmatrix} \mapsto \frac{1}{(j+1)s-j}
\begin{bmatrix}
-1/2 & -\sqrt 3/2 \\
-\sqrt 3 / 2 & 1/2
\end{bmatrix}
\begin{bmatrix}
x-\frac{j-(j+3)s}{4}\\
y-\frac{\sqrt 3j(1-s)}{4}
\end{bmatrix}
\]
The affine map $\Psi: \mathcal Y(I_j) \to \mathcal X$ and its inverse $\Psi^{-1}$ are given by
\[
\Psi(x,y,s)=(\psi_s(x,y), R(s)), \quad \Psi^{-1}(x,y,s)=(\psi_s^{-1}(x,y), \frac{1+sj}{1+s(j+1)}).
\]
By following the same computation in base case 1, we verify the renormalization theorem for $s \in I_1$ and $I_2$.

\subsection{Induction I}\label{ind1}
Consider  the map $T: [2/3, 1) \to [1/2, 1)$ given by
\[
T(s) = \frac{2s-1}{s} 
\]
and its inverse map 
\[
T^{-1}(s)=\frac{1}{2-s}. 
\]
The map $T^{-1}$ has action 
\[
[\frac{3}{4}, \frac{4}{5}) \mapsto [\frac{4}{5}, \frac{5}{6}) \mapsto \cdots \mapsto [\frac{n-2}{n-1}, \frac{n-1}{n}) \mapsto [\frac{n-1}{n}, \frac{n}{n+1}) \mapsto \cdots.
\]
It follows that 
\[
[\frac{1}{2}, 1) = [\frac{1}{2},\frac{2}{3}) \cup [\frac{2}{3}, \frac{3}{4}) \cup \bigcup_{k=0}^\infty T^{-k} [\frac{3}{4}, \frac{4}{5}).
\]

\begin{lemma}[Induction I]
If the Theorem \ref{main} is true for $t \in [\frac{n-1}{n}, \frac{n}{n+1})$ for $n \geq 4$, then it also holds for $s=T^{-1}(t) \in [\frac{n}{n+1}, \frac{n+1}{n+2})$.
\end{lemma}
\begin{figure}[h]
\centering
\includegraphics[scale=0.43]{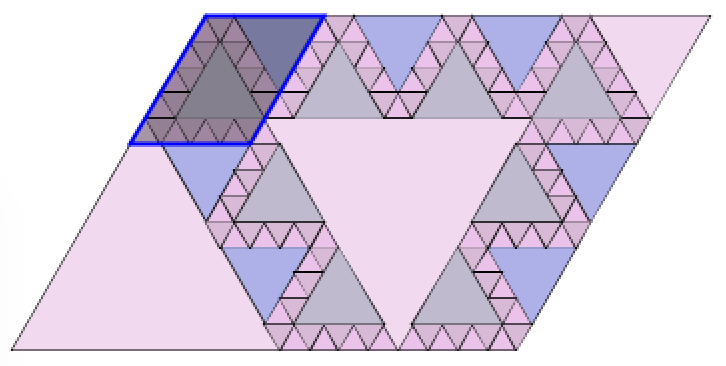}
\includegraphics[scale=0.43]{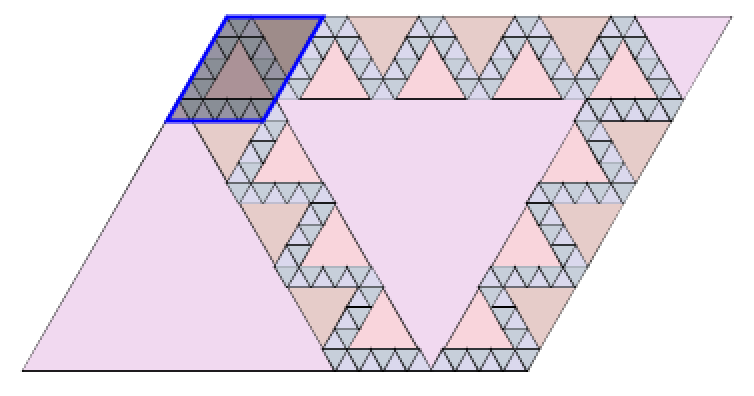}
\caption{$Y_t$ for $t=13/17$ and $Y_s$ for $s=T^{-1}(t)=17/21$} 
\includegraphics[scale=0.43]{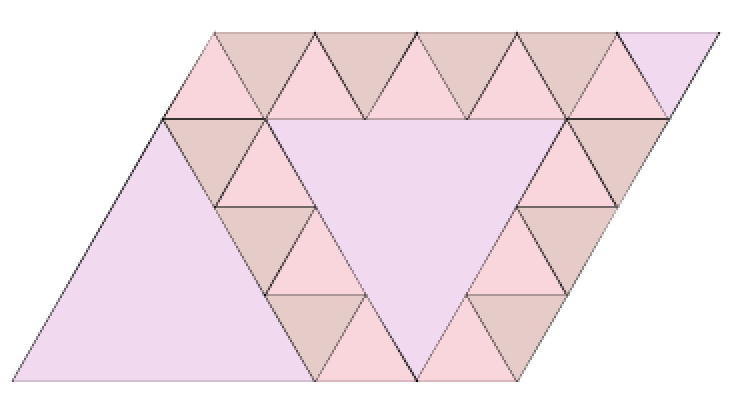}
\caption{$\Delta_u$ for $u=R(s)=R(t)=4/5$}
\end{figure}

It is equivalent to prove that the following diagram commutes. 
\[
\begin{CD}
Y_s @>f_s|_{Y_s}>> Y_s\\
@VV \xi_s V @VV \xi_s V\\
Y_t @>f_t|_{Y_t}>> Y_t
\end{CD}
\]
The map $\xi_s$ is a homothety (scaling and translation) with the scaling ratio $1/s$ such that $Y_s \mapsto Y_t$. The explicit formula of $\xi_s$ is given by
\[
\xi_s(x,y) = \frac{1}{s}(x-x_s, y-y_s) + (x_t, y_t),
\]
where $(x_s, y_s)=(\frac{s}{2}-1, \frac{\sqrt 3 s}{2})$ is the top left vertex of $X_s$ and the point $(x_t, y_t)$ is defined similarly. 

The idea of the proof is to modify the map $f_s$ in order to shorten the return time for the point $p \in Y_s$, and then show that the modified map of $f_s$ is conjugate to to the map $f_t$ via a similarity. Therefore, the statement of the first return map follows directly. To see this, we parametrize each element of the partition $\mathcal A_s$ and its translation vector. 

\begin{proof}
Let $\mathcal A_s=\{A_i(s)\}_{i=1}^m$ be a partition on $X_s$ which determine the map $f_s$ and $V_s=\{V_i(s)\}_{i=1}^m$ be a set of translation vectors such that 
\[
f_s(x)=x+V_i(s) \quad \mbox{for $x\in \Int(A_i)$}.
\]
According to previous section, we already know that there are 12 maximal domains $D_i$ in the fiber bundle $\mathcal X$ over the interval $[1/2, 1)$ and the maximal domain $D_7$ degenerates on the interval $[2/3,1)$. It follows that, there are 11 elements $A_i(s)$ for the partition $\mathcal A_s$ for $s\in [2/3, 1)$. We parametrize these 11 maximal domains $D_i$ and their translation vectors $V_i$ by the parameter $s \in (2/3, 1)$ for $i=\{0, 1, \cdots, 11\}$ and $i\neq 7$. 
\begin{figure}[h]
\centering
\includegraphics[scale=0.57]{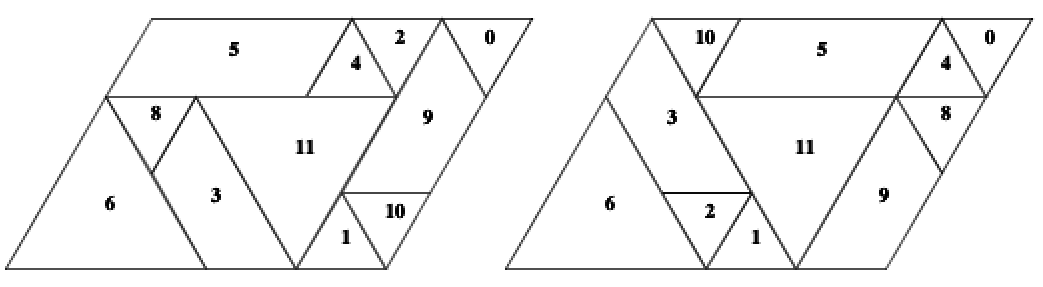}
\includegraphics[scale=0.55]{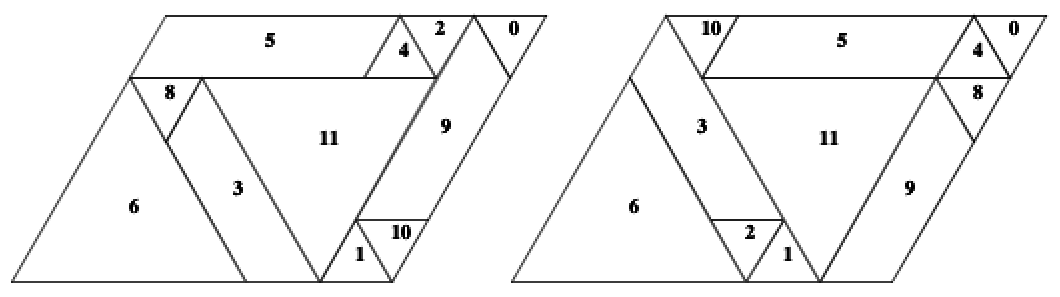}
\caption{The partition $\mathcal A_t$ and $\mathcal B_t$ of the triple lattice PET $f_s$ for $t=13/17$ (top); The partition $\mathcal A_s$ and $\mathcal B_s$ of the triple lattice PET $f_s$ for $s=17/21$ (bottom); }
\label{param}
\end{figure}

Let $(m_0,m_1, n_0, n_1)\in \Z^4$ be a coefficient tuple of a maximal domain $D_i$ in $\mathcal X$ (Section \ref{fiber}). Recall that for every point $(x,y,s)\in D_i$, 
\[
F(x,y,s)=(x,y,s)+(m_0s+m_1, (n_0s+n_1)\sqrt 3, 0).
\]
Let 
\[
a_s=1-s \quad \mbox{and}  \quad b_s=s-a_s=2s-1.
\]
Note that 
\[
a_t=1-t=1-\frac{2s-1}{s} \quad \Rightarrow \quad a_t=\frac{a_s}{s}, 
\]
\[
b_t= 2t-1=\frac{3s-2}{s} \quad \Rightarrow \quad b_t=\frac{b_s-a_s}{s}.
\]
According to the translation vectors (or coefficient tuples), we can divide the elements $A_i(s) \in \mathcal A_s$ for $i=\{0,1, \cdots, 11\}$ and $i\neq 7$ into three different classes:
\begin{enumerate}
\item The element $A_i(s)$ is fixed. When $i=0,6,11$, the set $A_i(s)$ is a trivial periodic tile given by the map $f_s$. Each $A_i(s)$  is an equilateral triangle. The parametrization of each triangle is listed below:
\[
A_0(s): \quad v= (\frac{1}{2}s+1, \frac{\sqrt 3}{2} s), \quad v-(2a_s, 0), \quad v-a_s(1, \sqrt 3),
\]
\[
A_6(s): \quad v=(-\frac{1}{2}s-1,-\frac{\sqrt 3}{2}s), \quad v+(2b_s,0), \quad v+b_s(1, \sqrt 3), 
\]
\[
A_{11}(s): \quad v=(-\frac{1}{2}s, (-\frac{3}{2}s+1)\sqrt 3), \quad v+(2b_s,0), \quad v+b_s(1, \sqrt 3).
\]
We restrict our attention on the parametrization of the elements in $\mathcal A_s$ which belongs to the subset 
\[
\bar{X}_s = X_s\setminus \bigcup_{i=0,6,11} A_i(s).
\]

\item For $i \in \{1, 4\}$, the translation vector $V_i(s)$ of $A_i(s)$ are of the form $2a_1\omega$ for $\omega$ is a unit vector in the direction of cubic root of unity. 
\[
V_1(s)= a_s(s)(2,0) \quad \mbox{and} \quad V_4(s)=-a_s(2,0).
\]
The element $A_i(s)$ are equilateral triangles of side length $2a_s$. 
\[
A_1: \quad v_0=(\frac{1}{2}s+1, \frac{\sqrt 3}{2}s), \quad v_0+(2a_s,0), \quad v_0+a_s(1, \sqrt 3).
\]
\[
A_4: \quad v_0=(\frac{7}{2}s-2, (\frac{3}{2}s-1)\sqrt 3), \quad v_0+(2a_s,0), \quad v_0+a_s(1, -\sqrt 3).
\]
Since $a_i(t)=\frac{1}{s}a_i(t)$, we have
\[
V_i(t)=\frac{1}{s}V_i(s),
\]
and $A_i(s), A_i(t)$ are same up to similarity with a scaling factor $1/s$. 

\item For $i\in \{2,8,10\}$, let $\Lambda_s$ be the lattice generated by the sides of the parallelogram $X_s$. The translation vector $V_i(s)$ are of the form $2a_s\omega \mod \Lambda_s$ where $\omega$ is a vector in the direction of cubic root of unity.  More precisely, 
\[
V_2(s)=a_s(-1,\sqrt 3) \mod \Lambda_s, \quad V_8(s)=a_s(-2,0) \mod \Lambda_s
\]
\[
V_{10}(s)=a_s(-1,-\sqrt 3) \mod \Lambda_s.
\]
Therefore, 
\[
V_i(t) = \frac{1}{s} V_i(s), \quad \mbox{for $i=2,8,10$}. 
\]
Each $A_i(s)$ is an equilateral triangle with a top horizontal side. The side length of each triangle is $2a_s$. The top left vertex $u_i$ of $A_i(s)$ is provided here:
\[
u_2=(\frac{5}{2}s-1,\frac{\sqrt 3}{2}s), \quad u_8=(\frac{3}{2}s-2,(\frac{3}{2}s-1)\sqrt 3), \quad u_{10}=(\frac{1}{2}s, (-\frac{3}{2}s+1)\sqrt 3).
\] 

\item For $i\in \{3,5,9\}$, $A_i(s)$ is a quadrilateral with translation vector $2a_s\omega$ where $\omega$ is a unit vector in the direction of cubic root of unity. The vertices and translation vector $V_i(s)$ of $A_i(s)$ are given as follows for $i=3,5$ or $9$:
\[
A_3(s): \quad
v_0=(\frac{7}{2}s-3,-\frac{\sqrt 3}{2}s ), \quad v_1=v_0-(2a_s,0), \quad v_0+b_s(-1,\sqrt 3)
\]
\[
v_1+(b_s-a_s)(-1,\sqrt 3), \qquad \mbox{ $V_3(s)=a_s(-1,\sqrt 3).$}
\]
The set $A_5(s)$ is a parallelogram determined by the vectors 
\[
 v_0+(2b_s,0), \quad v_0+a_s(-1,-\sqrt 3) \quad \mbox{where $v_0=(\frac{1}{2}s-1, \frac{\sqrt 3}{2}s)$}.
\]
The translation vector $V_5(s)=a_s(1,0)$. The vertices of $A_{9}(s)$ are given as follows:
\[
v_0=(\frac{5}{2}s-1, \frac{\sqrt 3}{2}s), \quad v_0+b_s(-1, \sqrt 3), \quad v_0+a_s(1, -\sqrt 3 s), 
\]
\[
v_1= (b_s-a_s)(-1,-\sqrt 3), \qquad \mbox{$V_9(s)=a_s(-1,-\sqrt 3).$}
\]
Consider the rhombus  $Z_i(s) \subset A_i(s)$ of side length $2a_s$ for $i=3, 5,10$ as shown below in blue in Figure \ref{reduct}. Note that $A_i(s) \setminus Z_i(s)$ and $A_i(t)$ are same up to a scaling factor $s$. Denote the translation image $Z_i(s)+V_i(s)$ by $Z'_i(s)$. Note that the sets $Z'_i(s)$ and $Z_i(s)$ are adjacent in $A_i(s)$. 

Let $e_i$ be the adjacent edge between the rhombi $Z_i$ and $Z'_i(s)$. Collapse the rhombus $Z'_i$ towards $e_i$ (Figure \ref{collapse}), and we obtain a new set $\bar X'_s$ such that $X_t$ and $\bar X'_s$ are same up to a scaling by $1/s$. Moreover, there is an induced partition $\mathcal A'_s=\{A'_i\}_{i=0}^{10}$ of $\bar X'_s$ (Figure \ref{collapse}) i.e.
\[
A'_i(s)=\left\{
\begin{array}{l l}
A_i(s) & \quad \mbox{if $i=8$}\\
A_i(s)-(2a_s,0) & \quad \mbox{if $i=2, 4$}\\
A_i(s)+a_s(1, \sqrt 3)-(2a_s,0) & \quad \mbox{if $i=1,10$}\\
A_i(s)\setminus Z'_i(s)-V_i & \quad \mbox{if $i=3,5,9$}.
\end{array}
\right.
\]
\begin{figure}[h]
\centering
\includegraphics[scale=0.47]{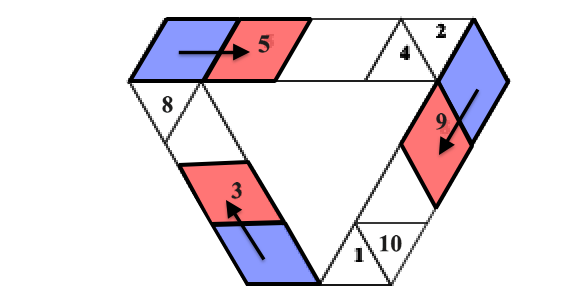} 
\caption{the set $Z_i(s)$ (blue) and its image $f_s(Z_i(s))$ (red) in $\bar X_s$}  
\label{reduct}
\end{figure}

\begin{figure}[h]
\centering
\includegraphics[scale=0.45]{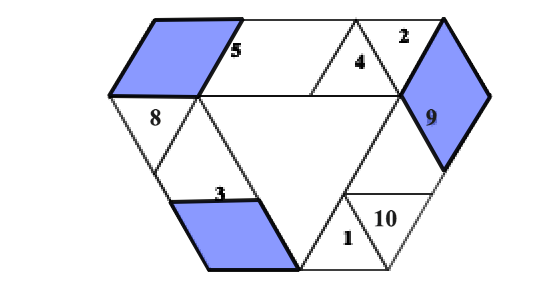}
\caption{The set $\bar X'_s$ for $s=17/21$ is similar to $X_t$ for $t=13/17$.}
\label{collapse}
\end{figure}
\end{enumerate}

We extend the map $\xi_s$ to the set $\bar X'_s$. Thus,  we have
\[
\bar X'_t = \xi_s(\bar X'_s) \quad \mbox{and} \quad A_i(t) = \xi_s(A'_i(s)).
\]
It follows that the return time for the points in $p \in Y_s$ will be equal to or longer than the return time of the  point $\xi(p) \in Y_t$. Since the vectors $V_i(s)$ and $sV_i(t)$ are same up to some translations by the vectors in $\Lambda_s$ or $\Lambda_t$, the images $f_s|_{Y_s}(p)$ and $f_t|_{Y_t}(\xi(p))$ must be in the corresponding positions. Hence, we have shown the first return maps $f_s|_{Y_s}$ and $f_t|_{Y_t}$ conjugates by the map $\xi$.
\end{proof}

\subsection{Induction II}
Suppose $s \in (0,\frac{1}{2})$. Let $T: (0,1) \to (0,1)$ be the map given by 
\[
T(s) = \frac{s}{1+s} \quad \mbox{and} \quad T^{-1}(s) = \frac{s}{1-s}.
\]
The map $T$ has the action 
\[
[\frac{1}{2}, 1) \mapsto [\frac{1}{3}, \frac{1}{2}) \mapsto \cdots \mapsto [\frac{1}{n+1}, \frac{1}{n}) \mapsto [\frac{1}{n+2}, \frac{1}{n+1}) \mapsto \cdots. 
\]
Therefore, 
\[
(0,1) = \bigcup_{n\geq 0} T^n([\frac{1}{2}, 1)).
\]

Suppose $s \in (0,\frac{1}{2})$. Let $a_1 \geq 2$ be the first non-zero digit of the continued fraction expansion of $s$, then we have the relation 
\[
R (s) = T^{1-a_1}(s).
\] 
Define $\Diamond^0_k$ (Figure 17) to be the rhombus of side length $2s$ centered at  
\[
((1-2k)s +1,0) \quad \mbox{for} \quad k=1,\cdots, a_1-1.
\]
The sides of $\Diamond_k$ are parallel to the vectors $(2,0)$ and $(s, \sqrt 3 s)$. Define 
\[
Y'_s = X_s \setminus \bigcup_{k=1}^{a_1-1} \Diamond_k^0.
\]
Then, we obtain the subset $Y_s$ in the Renormalization Theorem by applying the cut-and-paste operation $\mathcal{OP}$ (Section 2.2) on the parallelogram $X_s$. 

Similarly, we define $\Diamond^1_k$ as the rhombus of side length $2s$ centered at 
\[
(2ks-1,0) \quad \mbox{for} \quad  k=1, \cdots, a_1-1 .
\]
The sides of $\Diamond^1_k$ are parallel to the vectors $(2,0)$ and $(-s, \sqrt 3 s)$. We call the rhombus $\Diamond^0_k$ or $\Diamond^1_k$ the \textit{central tile}. Let $\Diamond_{s}$ be the union of all central tiles in $X_s$, i.e.
\[
\Diamond_s = \bigcup_{k=1}^{a_1-1} (\Diamond^0_k \cup \Diamond^1_k).
\]
The union $\Diamond_s$ remains the same under the triple lattice PET $f'_s$. Define $\bar Y'_s$ to be the complement 
\[
\bar Y'_s = X_s \setminus \Diamond_{s}.
\] 
\begin{figure}[h]
\centering
\includegraphics[scale=0.56]{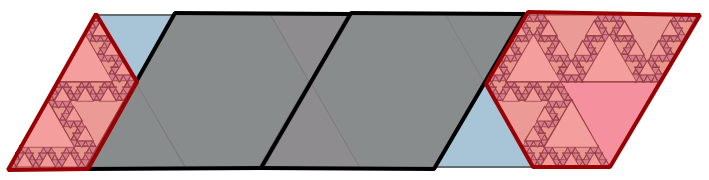}
\includegraphics[scale=0.55]{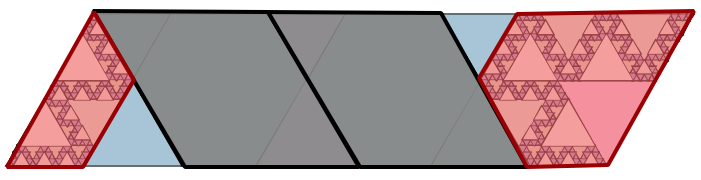}
\caption{$\Diamond^0_k$ (shaded) on the left, $\Diamond^1_k$ (shaded) on the right for $k=1,2$ and $\bar Y'_s$ (red)}
\end{figure}

\begin{figure}[h]
\centering
\includegraphics[scale=0.55]{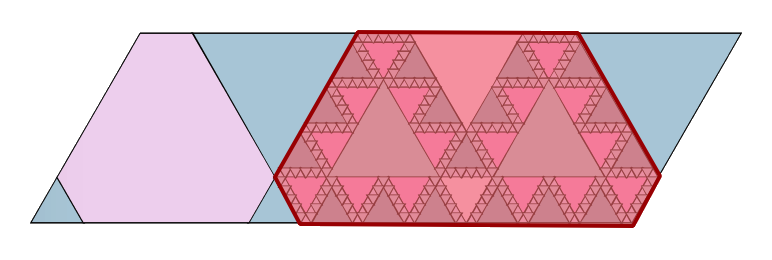}
\includegraphics[scale=0.57]{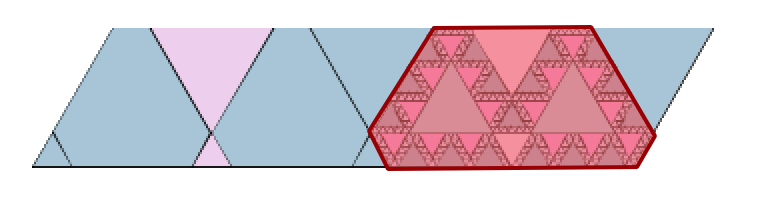}
\caption{ The set $\bar Y_t$ (red) for $t=21/58$ is shown on the left and $\bar Y_s$ (red) for $s=21/79$ is on the right. Here we provide a pictorial illustration of Induction II.}
\end{figure}

\begin{lemma}[Induction II]
If Theorem \ref{main} is true for some $u \in [\frac{1}{n}, \frac{1}{n-1})$ for $n \geq 2$, then it holds true for $s=T(u)\in [\frac{1}{n+1}, \frac{1}{n})$.
\end{lemma}

\begin{proof}
We prove the statement by studying the triple lattice map $f'_s$ and the map $g_s$ such that $f'_s = g^3_s$. By assumption, the renormalization is true for $u \in (\frac{1}{n}, \frac{1}{n-1})$. Since 
\[
R(u) = R(s) \quad \mbox{for $s=T(u)$},
\]
and every point is fixed under the first return map $f'_s|_{Y'_s}$, we want to show that the restricting maps $f'_s$ on $\bar Y'_s$ and $f'_u$ on $\bar Y'_u$ is conjugate via similarites. Then, by applying the cut-and-paste operation, we obtain the desired statement.  

Let $P_0$ and $P_1$ be polygons in $\R^2$. Let  $Q_i$ be a subset of $P_i$ such that $P_i \setminus Q_i$ has $k \geq 1$ connected components $U^i_j$ for $i=0,1$ and $j=0,\cdots, k$. We say two points $p \in P_0 \setminus Q_0$ and $q \in P_1 \setminus Q_1$ are \textit{partners} if they are in the corresponding positions relative to $Q_0$ and $Q_1$, i.e. there exists a piecewise similarity $\alpha: P_0 \to P_1$ such that $U^0_j \mapsto U^1_j$ for each $j$, and $q = \alpha(p)$. We show that the image of a pair of partners $(p_s, p_u)$ relative to $\Diamond_s$ and $\Diamond_u$ are also partners. 

Let $W^0_s$ and $W^1_s$ be the left and right subsets of $\bar Y'_s$ relative to the central tiles. If $p_s \in \bar Y'_s $ and $p_u \in \bar Y'_u$ are partners relative to $\Diamond_s$ and $\Diamond_u$, then the  associated piecewise similarity $\alpha: X_s \to X_u$ is given by 
\[
\alpha_s(p_s) = \left\{
\begin{array}{l l}
\frac{1}{1-s}(p_s - V^0_s)+ V^0_u \quad \mbox{if $p_s \in W^0_s$}\\
\frac{1}{1-s}(p_s - V^1_s) + V^1_u \quad \mbox{if $p_s \in W^1_s$}
\end{array}\right.
\] 
where $V^0_\omega$ and $V^1_\omega$ are the lower left and right vertices of $X_\omega$ respectively for $\omega\in \{s,u\}$.

Let the points $p_s \in X_s \setminus \Diamond_s$ and $p_u \in X_u \setminus \Diamond_u $ be partners.  Let $v_s$ and $v_t$ be the vectors in $(L_0)_s$ and $(L_0)_u$ (Section 1.3) such that $p_s + v_s \in (F_1)_s$ and $p_u+ v_u \in (F_1)_u$. Depending on the position of right and left respect to the diamonds, the vector $\nu =v_s - (1-s)v_u$ belongs to one of the following three case: 
\begin{enumerate}[1)]
\item  $\nu = (-s, \sqrt 3 s)$ if $p_s$ and $p_u$ belong to left triangle in $Y'_s$ and $Y'_u$ respectively. 
\item $\nu = (-2s, 0)$ if $p_s$ and $p_u$ belongs to the parallelogram right to the central tiles. 
\item $ \nu = (-2s,0)+(-s, \sqrt 3 s)$ otherwise.
\end{enumerate}
It follows that the image $g_s(p_s)$ and $g_u(p_u)$ are in the corresponding positions relative to $g_s(\Diamond_s)$ and $g_u(\Diamond_u)$.

\begin{figure}[h]
\centering
\includegraphics[scale=0.5]{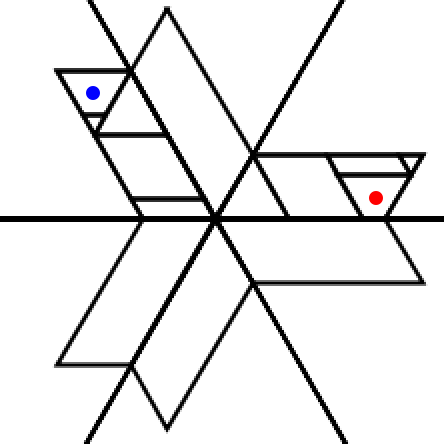}
\hspace{4em}
\includegraphics[scale=0.5]{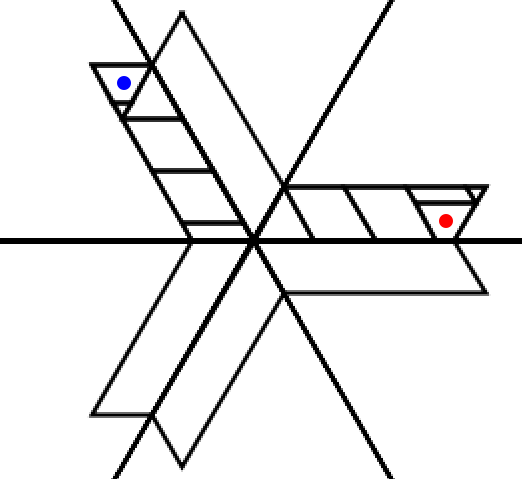}
\caption{Partners $(p_s, p_u)$ (red) and $(g_s(p_s), g_u(p_u))$ for $u=7/17$ and $s=T(u)=7/24$}
\end{figure}

Similarly, a pair of partners $(p_s \in (F_1)_s, p_u \in (F_1)_u)$ relative to $g(\Diamond_s)$ and $g(\Diamond_u)$ becomes partners relative to $g_s^2(\Diamond_s)$ and $g^2_u(\Diamond_u)$ under the maps $g_s$ and $g_u$. Let $\nu = v_s = (1-s) v_u$ where $v_s$ and $v_u$ are defined similarly as before. The vector $\nu=v_s-(1-s) v_u$ is 
\[
(s, -\sqrt 3 s), \quad -(s, \sqrt 3 s) \quad \mbox{and} \quad  (0,-2\sqrt 3 s) = -(s, \sqrt 3 s)+(s, -\sqrt 3 s)
\]
depending on whether or not $p_s, p_u$ and their image lie in the top or bottom of the $F_1$ and $F_2$ respectively. Same argument works for pair of partners in $(F_2)_s$ and $(F_s)_u$ for the vector $\nu = v_s - (1-s) v_u$ is one of the following forms depending on the positions of the points $p_s, p_u$ and their image in the parallelograms
\[
(2s, 0), \quad (s, \sqrt 3 s), \mbox{and} \quad (3s, \sqrt 3 s) = (2s,0) + (s, \sqrt 3 s).
\] 
It follows that the maps $f_s$ and $f_u$ restricting on the sets $\bar Y_s$ and $\bar Y_u$ are conjugate via a map of similarity. Therefore, we have shown the inductive step for the renormalization on the interval $(0,\frac{1}{2})$.
\end{proof}

\section{Symmetry}
In this section, we analyze the symmetries of the periodic pattern $\Delta_s$ for $s \in [8/13, 13/21]$, which will help us to understand the properties of the limit set $\Lambda_\phi$ later.

\subsection{Symmetry 1}\label{sym1}
Let $s \in [8/13, 13/21]$ and $A, B$ and $C$ be isosceles trapezoids in  $X_s$ as illustrated in Figure \ref{sym}. In this section, we will show that the periodic tiles in $A, B$ and $C$ are same up to rotation and reflection. 

Here we give a precise description of the trapezoids $A, B$ and $C$. The vertices of trapezoid $A$ are 
\[
v_0=(-1+\frac{1}{2}s, \frac{\sqrt 3}{2}s), \quad v_1=v_0+s(1, -\sqrt 3)
\]
\[
v_0-(1-s)(1, \sqrt 3), \quad \mbox{and} \quad  v_1-2(1-s)(1,0).
\]
The quadrilateral $A$ is an isosceles trapezoid of base angle $\pi/3$ and side length $2(2s-1)$, $2(1-s)$ and $2s$. For the trapezoid $B$ and $C$, let $\mu_s$ be the map of rotation by $\pi/3$ about the upper left vertex of the parallelogram $X_s$, and $\nu_s$ be the map of reflection about the vertical line $x=\frac{3}{2}s-1$. The trapezoids $B$ and $C$ are obtained by the formulas
\[
B=\mu_s(A) \quad \mbox{and} \quad C=\nu_s(A).
\]
\begin{figure}[h]
\centering
\includegraphics[scale=0.5]{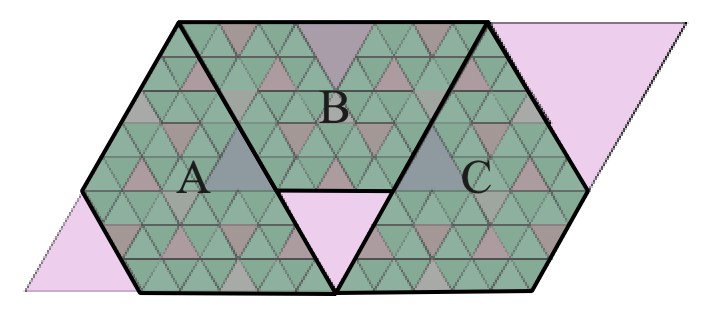}
\caption{$L_s$ (lightly shaded triangle) and $M_s$  (transparent triangle) for $s=8/13$}
\label{sym}
\end{figure}

The following lemma shows that there exists rotational symmetry between periodic pattern restricting in $A$ and $B$, and a reflectional symmetry between the periodic patterns in subsets $A$ and $C$. 
\begin{lemma}
Suppose $s \in [8/13, 13/21]$. Then, the following two equations are satisfied on $M_s$ and $N_s$ respectively:
\begin{enumerate}
\item   
\[
\mu_s \circ f_s|_{A} \circ \mu_s^{-1} = f_s^{-1}|_{B}.
\]
\item
\[
\nu_s \circ f_s|_{A} \circ \nu_s^{-1} = f_s^{-1}|_{C}.
\]
\end{enumerate}
\end{lemma}

The proof follows the similar scheme as the one for the base case of the Renormalization theorem (Section \ref{base}). Here we explain the differences. 

Recall that $\mathcal X$ is the fiber bundle over $[1/2,1]$. We define $\mathbf A \subset \mathcal X$ to be the polyhedron over the interval $[8/13, 13/21]$ such that  the fiber above $s$  is the isosceles trapezoid $A_s$, i.e.
\[
\mathbf A = \{(x,y,s)| (x,y) \in A_s  \mbox{ and }  s \in [8/13, 13/21]\}.
\]
Let $F: \mathcal X \to \mathcal X$ be the fiber bundle map on $\mathcal X$. Let $\mathcal S$ be a subset of $\mathcal X$. A \textit{maximal return domain} in $\mathcal S$ under the map $F$ is a maximal subset of the return map $F|_{\mathcal S}$ where $F|_{\mathcal S}$ is entirely defined and continuous. 

With computer assistance, we find that there are 12 maximal return domains $\alpha_i$ in $\mathbf A$, each of which is a convex polytope, i.e.
\[
\mathbf A = \alpha_0 \cup \alpha_1 \cdots \cup \alpha_{10}. 
\]
Moreover, the vertices of each maximal domain $\alpha_i$ are of the format 
\[
\begin{bmatrix}
(m_0s_0+n_0)/2\\
(p_0s_0+ q_0) \sqrt 3/2\\
s_0
\end{bmatrix}
\begin{bmatrix}
(m_0s_1+n_0)/2\\
(p_0s_1+ q_0) \sqrt 3/2\\
s_1
\end{bmatrix}
\cdots
\begin{bmatrix}
(m_{k-1}s_0+n_{k-1})/2\\
(p_{k-1}s_0+ q_{k-1}) \sqrt 3/2\\
s_0
\end{bmatrix}
\begin{bmatrix}
(m_{k-1}s_1+n_{k-1})/2\\
(p_{k-1}s_1+ q_{k-1}) \sqrt 3/2\\
s_1
\end{bmatrix}.
\]
where $m_j,n_j, p_j,q_j \in \Z$ for $0<j <k$, $s_0=8/13$ and $s_1=13/21$. Same as the fiber bundle map $F: \mathcal X \to \mathcal X$, the first return map $F|_{\mathbf A}$ is an piecewise affine map on each maximal return domain $\alpha_i$ partitioning $\mathbf A$. Let  $\alpha$ be a maximal return domain in $\mathbf A$, for $(x,y,z)\in \alpha$, the map $F|_{\mathbf A}$ has action
\[
(x,y,z)\mapsto (x,y,z)+(\frac{az+b}{2}, \frac{cz+d}{2}\sqrt 3,0), \quad a,b,c,d\in \Z.
\]
We call $(a,b,c,d)$ the \textit{coefficient tuple} of the maximal return domain. For convenience, we write maximal return domains of the map $F|_{\mathbf A}$ along with their coefficient tuples in matrix format: 
\[
\begin{bmatrix}
m_0 & n_0 & p_0 & q_0\\
m_1 & n_1 & p_1 & q_1\\
\vdots & \vdots & \vdots & \vdots\\
m_{k-1} & n_{k-1} & p_{k-1} & q_{k-1}\\
a & b & c& d
\end{bmatrix}.
\]
The list of all maximal return domains partitioning polytope $\mathbf A$ is provided in Section \ref{data_sym1}. 

Define the affine map $\Upsilon: \mathcal X \to \mathcal X$ by piecing together all the rotations $\mu_s$ for $s \in [8/13, 13/21]$, i.e.
\[
\Upsilon(x,y,s) = (\mu_s(x,y), s)
\]
The rest of the calculation is the same as the one in Section \ref{base}.

\subsection{Symmetry 2}
\begin{figure}[h]
\centering
\includegraphics[scale=0.5]{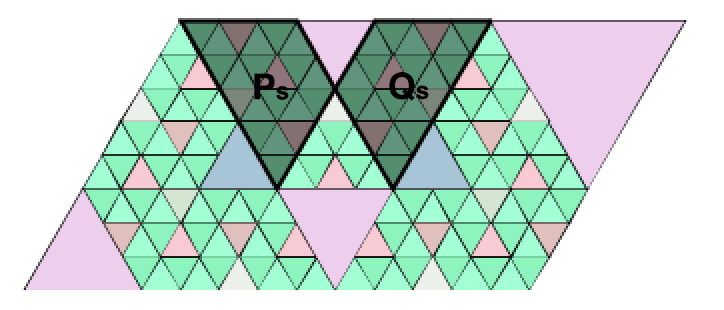}
\caption{An illustration of reflectional symmetry of periodic patterns}
\end{figure}

Let $P_s$ be a isosceles trapezoid of the symmetric piece $M_s$ with base angle $\pi/3$ such that $P_s$ has vertices
\[
\begin{bmatrix}
\frac{1}{2}s-1\\
\frac{\sqrt 3}{2}s
\end{bmatrix}
\begin{bmatrix}
\frac{9}{2}s-3 \\
\frac{\sqrt 3}{2}s
\end{bmatrix}
\begin{bmatrix}
\frac{3}{2}s-1\\
(\frac{7}{2}s-2)\sqrt 3
\end{bmatrix}
\begin{bmatrix}
-\frac{s}{2}\\
(\frac{3}{2}s-1)\sqrt 3
\end{bmatrix}.
\]
Let $\iota_s: X_s \to X_s$ be the reflection about the vertical line $x=\frac{3}{2}s-1$. Define the set $Q_s$ as
\[
Q_s = \iota(P_s).
\]
The following lemma states that the periodic tiles in $P_s$ and the ones in $Q_s$ are same up to reflection $\iota_s$. 
\begin{lemma} 
Suppose $s \in [8/13, 13/21]$, then
\[
\iota_s \circ f_s^{-1}|_{P_s} \circ \iota_s = f_s|_{Q_s}.
\]
\end{lemma}
We apply the same method as the proof of Lemma 4.1. Define $\mathcal Q$ be the fiber bundle over $[8/13, 13/21]$
\[ 
\mathcal Q = \{(x,y,s)| (x,y) \in Q_s \mbox{ and } s \in [8/13, 13/21]\}.
\]
The list of maximal return domains in the polyhedra $\mathcal Q$ along with their coefficient tuples are provided in Section 7. 
 
\section{Limit Set $\Lambda_\phi$ for $\phi=\frac{\sqrt 5-1}{2}$}
In this section, we describe one application of our main theorem. We show that $\Delta_\phi$ is dense. We explore the properties of the limit set $\Lambda_\phi$ where $\phi$ is the only fixed point under the renormalization map $R$. Denote symmetric trapezoids defined in Section \ref{sym1} as $A, B$ and $C$ for $s=\phi$. We call the trapezoids $A, B$ and $C$ the \textit{fundamental trapezoids}. 

\subsection{Shield Lemma}
 Let $l_0$ be the line such that the intersection $A\cap B$ belongs to $l_0$. We say that a polygon $P$ \textit{abuts} on $l_0$ if $P$ has an edge contained in $l_0$. We call such segment a \textit{contact}. 
\begin{lemma}[Shield]
Let $V$ be the top left vertex of $X_s$. There is a sequence of periodic tiles $\{P_n\}_{n=0}^\infty$ satisfying the following properties:
\begin{enumerate}
\item  Each element $P_n$ is an equilateral triangle abutting the line segement $l_0$.
\item The sequence occur in a monotone decreasing size. The periodic tile $P_{n+1}$ is similar to $P_n$ of scaling factor $\phi$ for all $n \geq 0$.
\item The periodic tiles in the sequence, from largest to smallest, move towards the point $V$. 
\item For any point $p \in l_0 \setminus V$, there must be a periodic tile $P_n$ in the sequence whose contact of positive length contains $p$.
\end{enumerate}
\end{lemma}

\begin{proof}
Let $P_0$ be a fixed periodic tile of $X_\phi$ under the map $f_s$ (Figure \ref{fig:shield}). The tile $P_0$ is an equilateral triangle of side length $4s-2$ with a horizontal top side. The bottom vertex has coordinate $(3s/2-1, -\sqrt 3/2)$. The triangle $P_0$ abuts the line $l_0$. Recall the map of similarity $\psi_s$ in Theorem \ref{main} taking the subset $Y_s$ to $X_{R(s)}$. For convenience, we write $\zeta$ as the inverse $\psi^{-1}_\phi$. The scaling factor of the similarity $\zeta$ is $\phi=\frac{-1+\sqrt 5}{2}$. 

Consider a sequence $\{P_n \}_{n=0}^\infty$ such that 
\[
P_n = \zeta^n(P_0).
\]
Since $\phi$ is fixed by the renormalization map $R$, each $P_n$ is a periodic equilateral triangle of $X_\phi$. Note that the line $l_0$ is invariant under the map $\zeta$ so that each periodic $P_n$ abuts to $l_0$.  Moreover, the triangles $P_n$ in the sequence point towards different side of $l_0$ alternatively. Therefore, we have shown statement 1, and statement 2 and 3 follow directly. 

According to the renormalization and properties of the map $\xi$
\[
P_i \cap P_j = \left\{
\begin{array}{l l}
\emptyset \quad & \mbox{if $j\neq i\pm 1$}\\
v \quad & \mbox{otherwise}
\end{array}
\right.
\]
where $v$ is vertex for both $P_i$ and $P_j$. Therefore, for any point $p \in l_0 \setminus V$, there must exists an element $P_n=\zeta^n(P_0)$ for some $n \geq 0$ such that $P_n$ has a side of positive length on $l_0$ containing  $p$. 
\begin{figure}[h]
\begin{center}
\centering
\includegraphics[scale=0.45]{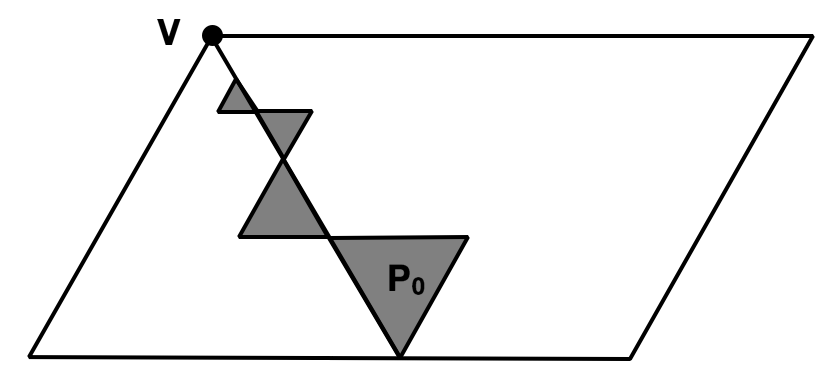}
\caption{The sequence of periodic tiles in Lemma 5.1}
\label{fig:shield}
\end{center}
\end{figure}
\end{proof} 

\begin{corollary}\label{contact}
Every point $p \in l_0\setminus V$ is contained in the edge of a periodic triangle given by the map $f_\phi$. 
\end{corollary}

Note that the union of the fundamental trapezoids $A\cup B\cup C$ contains all periodic tiles $\Delta_p \in \Delta$, except for three fixed periodic triangles, so 
\[
\Lambda_\phi \subset A \cup B \cup C. 
\]
Let $S \subset X_s$ be a subset and $\Lambda_\phi(S)=\Lambda_\phi \cap S$.  The Shield lemma implies that 
\[
\Lambda_\phi(A) \cap \Lambda_\phi(B) = \{\xi^{-n}(V_0)\}_{n=1}^\infty
\]
where $V_0$ is the bottom vertex of $P_0$.  
\subsection{Substitution Rule}
In this section, we provide a precise description of the substitution rule introduced in Section \ref{main}. Let $M$ be an isosceles trapezoid such that the base angles of $M$ are $\pi/3$. The parallel sides of $M$ are of side length $2s$ and $2(1-s)$, and two non-parallel ones are of length $2(2s-1)$. The trapezoid $M$ is substituted by three similar isosceles trapezoids $M_1, M_2$ and $M_3$  of scaling factor $\phi, \phi^2$ and $\phi$ respectively (Figure \ref{fig:sub_rule}). The non-parallel sides  of $M$ become the longest parallel sides of $M_1$ and $M_3$, and the top side of $M$ becomes the longer parallel side of $M_2$. Note that the trapezoids $M_1$ and $M_3$ are same up to reflection about the vertical line passing through the midpoint of the parallel sides of $M$. 
\begin{figure}[h]
\centering
\includegraphics[scale=0.45]{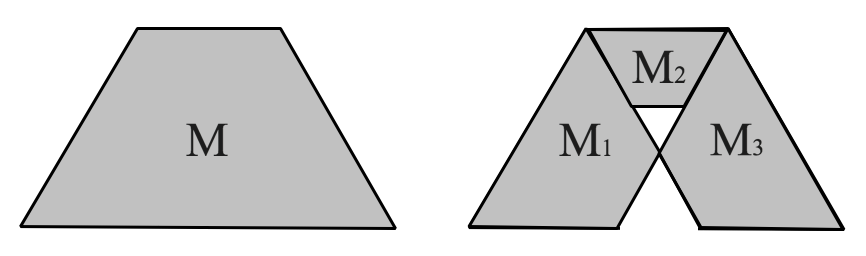}
\caption{Substitution rule}
\label{fig:sub_rule}
\end{figure}

We apply the substitution on each fundamental trapezoid. By symmetries, we can restrict our focus to the substitution of the isosceles trapezoid $A$. Let $A_1$, $A_2$ and $A_3$ be the three substituted trapezoids of $A$ as shown in Figure \ref{fig:first_sub}. We have the following observations:
\begin{figure}[h]
\centering
\includegraphics[scale=0.5]{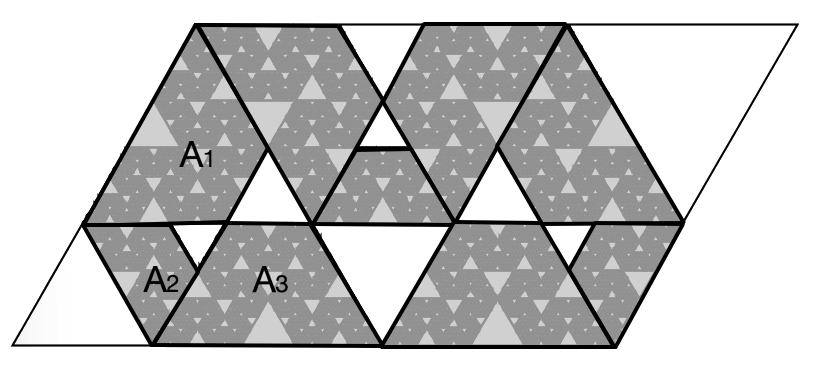}
\caption{The substitution on fundamental trapezoids $A$, $B$ and $C$}
\label{fig:first_sub}
\end{figure}
\begin{enumerate}
\item Denote the equilateral triangle bounded by sides of $A, B$ and $C$ by $P_0$. The triangle $P_0$ is a fixed periodic tile under the map $f_\phi$. Recall the map $\xi=\psi_\phi^{-1}$ where $\psi_\phi$ is the map of similarity for renormalization. Let $\iota_1$ be the reflection about the line $l_1: y=\frac{\sqrt 3}{3}(x+1-s)$ where $l_1=\xi(l)$ and $l: x=\frac{3}{2}s-1$ is the line of reflectional symmetry. The set $A\setminus \bigcup_{i=1}^3 A_i$ is composed of two triangles:
\[
\xi(P_0) \quad \mbox{and} \quad \iota_1 \circ \xi^2(P_0).
\]
Since the triangle $P_0$ is a periodic tile of $f_\phi$, according to renormalization and symmetry of $f_s$, both triangles $\xi(P_0)$ and $\iota_1 \circ \xi^2(P_0)$ are periodic tiles  in containing in the fundamental trapezoid $A$. 
\item Let $S \subset X_s$ and $\Delta(\mathcal S) = \Delta_\phi \cap S$ be the restriction of the periodic pattern $\Delta_\phi$ on $S$. Then $\Delta(A_i)$ is same as $\Delta(A)$ up to a map of similarity for all $i\in \{1,2,3\}$. To see this, recall that $\mu_s$ is the map of rotation about the top left vertex of $X_s$ by $\pi/3$. We have the relations
\[
A_1 = \mu_\phi \circ \xi(A), \quad A_2 = \mu_\phi \circ \xi^2(A) \quad \mbox{and} \quad A_3 =\iota_1(A_1).
\]
\end{enumerate}

\begin{lemma}
The periodic pattern $\Delta_\phi$ is dense.
\end{lemma}
\begin{proof}
By symmetries, we restrict our focus to the fundamental trapezoid $A$. We show that any point $p \in A$ is within $\epsilon$ of some periodic tiles. According to previous discussion, the fundamental trapezoid $A$ is partitioned into a finite union of periodic tiles and a finite union of trapezoids which are same as $A$ up to similarities. The similarities are contractions by the factor $\phi^n$ for $n \geq 0$.  We call a trapezoid which is similar to $A$ and obtained by substitutions an $\epsilon$-\textit{patch} if its longest side has length $l_e \leq \epsilon$.  

By iterating the substitution, for any $p \in A$, we have $p \in \Delta_p$ for some periodic tile $\Delta_p$ or $p$ must lie in some $\epsilon$-patch $K$. For the first case, we are done. If $p \in K$ for some $\epsilon$-patch $K$, then $K$ must contain a periodic triangle by previous discussion. Therefore, the point $p$ lie within $\epsilon$ of the periodic tile. 
\end{proof}

\subsection{Proof of Theorem 1.7}
In this section, we show that the limit set $\Lambda_\phi$ can be obtained by a sequence of substitutions on the fundamental trapezoids. 

Let $\mathcal C_0=A\cup B\cup C$. The set $\mathcal C_{n+1}$ is defined inductively as the union of all trapezoids given by applying substitution rule on each trapezoid in $\mathcal C_{n}$ for $n \geq 0$. We call the trapezoids in $\mathcal C_n$ marked trapezoids. Define
\[
L\Lambda = \lim_{n\to \infty} \mathcal C_n.
\] 
Here we show that 
\[
L\Lambda = \Lambda_\phi.
\]
``$\subseteq$'' Take an arbitrary point $p \in L\Lambda$. Since the sequence of chains $\{\mathcal C_n\}_{0}^\infty$ is a nested family of unions of marked trapezoids, any neighborhood of $p$ intersects an infinite sequence of marked trapezoids. According to previous section, each marked trapezoid contains infinitely many periodic tiles. Therefore, we have $p \in \Lambda_\phi$.\\
``$\supseteq$'' Take a point $p \in \Lambda_\phi$, then every neighborhoods $U_p$ of $p$ intersects infinitely many periodic tiles. There must exists a trapezoids $M \in \mathcal C_{n_0}$ so that $p \in M$ and $M \cap U_p \neq \emptyset$. Otherwise, $p$ must be in some periodic tile, which leads to a contradiction. We want to show that $p$ belongs to the chain $\mathcal C_n$ for every $n \geq n_0$. Suppose there exists a marked trapezoid $M_n$ such that $p \in M_n \setminus M'$ for any marked trapezoid $M' \subset M$ and $M' \in \mathcal C_{n+1}$, then there is a neighborhood of $p$ intersecting only finitely many periodic tiles in $\Delta_\phi$. Contradiction. \qed

\subsection{Proof of Theorem 1.9}
Because of the reflection and rotational symmetries of the fundamental trapezoids $A$, $B$ and $C$, we restrict our focus on the substitutions in trapezoid $A$ in the chain $\mathcal C_0$ first. Let $\sigma_i$ be the similarity on the trapezoid $A$ to smaller trapezoids $A_i$ by substitution for $i=1,2$ and $3$. Let $\Lambda_\phi|_A$ be the limit set $\Lambda_\phi$ restricting on $A$. By construction, the limit set $\Lambda_\phi|_A$ is the attractor of the iterated function system given by the similarites $\{\sigma_i\}_{i=1}^3$, i.e. 
\[
\Lambda_\phi|_A = \bigcup_{i=1}^3 \sigma_i(\Lambda_\phi|_A).
\]
To compute the Hausdorff dimension of $\Lambda_\phi$, we check the open set condition (Section 2.5). The condition holds simply by taking the open set $V$ as the interior of $A$. Thus, according to Theorem 9.3 in \cite{FA}, we have the equation:
\[
2\phi^s+\phi^{2s} = 1.
\]
The Hausdorff dimension for $\Lambda_\phi$ in the polygon $A$ is 
\[
s=\frac{\log(-1+\sqrt 2)}{\log \phi}=1.83147\cdots.
\]
Same computation on the limit set $\Lambda_\phi$ restricting in the trapezoids $B$ and $C$. By elementary properties of Hausdorff dimension, we have
\[
1<dim_H(\Lambda_\phi)=\frac{\log(-1+\sqrt 2)}{\log \phi}=1.83147\cdots<2. \qed
\]

\section{The Computational Data}\label{data}
\subsection{Data of Maximal Domains}\label{max}
Here we list all the maximal domains $D_i$ of the fiber bundle $\mathcal X$ for the map $F$ and their coefficient tuples $\omega_i$ for $i=0,\cdots, 11$. To have a nicer form, we scale the $y$-coordinates by $1/\sqrt 3$.

{\tiny
\[
D_0:  \begin{bmatrix}
2/3\\
1/2 \\ 
1
\end{bmatrix}
\begin{bmatrix}
5/4 \\
1/4 \\
1/2
\end{bmatrix}
\begin{bmatrix}
1/4 \\ 
1/4 \\
1/2
\end{bmatrix}
\begin{bmatrix}
3/4 \\ 
-1/4 \\
1/2
\end{bmatrix}, \quad 
\omega_0 = [0,0,0,0];
\qquad 
D_1: \begin{bmatrix}
1/2 \\
-1/2 \\
1
\end{bmatrix}
\begin{bmatrix}
1/4 \\
1/4 \\
1/2
\end{bmatrix}
\begin{bmatrix}
3/4 \\
-1/4 \\
1/2
\end{bmatrix}
\begin{bmatrix}
-1/4 \\
-1/4 \\
1/2
\end{bmatrix}
\begin{bmatrix}
1/4 \\
-1/4 \\
1/2
\end{bmatrix}, \quad
\omega_1 = [2,-2, 0,0];
\]
\[
D_2: \begin{bmatrix}
3/4 \\ 
1/4 \\
1
\end{bmatrix}
\begin{bmatrix}
0 \\
1/6 \\
2/3
\end{bmatrix}
\begin{bmatrix}
1/8 \\
1/8 \\
1/2 
\end{bmatrix}
\begin{bmatrix}
-1/8 \\
-1/8\\
1/2
\end{bmatrix}, \quad \omega_2 = [0,-1,-2,1];
\qquad 
D_3: \begin{bmatrix}
-1/2 \\
1/2 \\
1
\end{bmatrix}
\begin{bmatrix}
1/2 \\
-1/2 \\
1
\end{bmatrix}
\begin{bmatrix}
-2/3 \\
-1/3 \\
2/3 
\end{bmatrix}
\begin{bmatrix}
-1/4 \\
-1/4 \\
1/2
\end{bmatrix}, \quad \omega_3 = [1,-1,-1,1];
\]
\[
D_4: \begin{bmatrix}
3/2 \\ 
1/2 \\
1
\end{bmatrix}
\begin{bmatrix}
-3/4 \\
1/4 \\
1/2 
\end{bmatrix}
\begin{bmatrix}
-1/4 \\
-1/4 \\
1/2
\end{bmatrix}
\begin{bmatrix}
-5/4 \\
-1/4 \\
1/2
\end{bmatrix}, \quad \omega_4 = [-2,2,0,0];
\qquad 
D_5: \begin{bmatrix}
3/2 \\
1/2 \\
1
\end{bmatrix}
\begin{bmatrix}
-1/2 \\
1/2 \\
1
\end{bmatrix}
\begin{bmatrix}
-3/4 \\ 1/4 \\ 1/2 
\end{bmatrix}
\begin{bmatrix}
-5/4 \\ -1/4 \\ 1/2
\end{bmatrix}, \quad \omega_5 = [-2,2,0,0];
\]
\[
D_6: \begin{bmatrix}
-1/2 \\ 1/2 \\ 1
\end{bmatrix}
\begin{bmatrix}
1/2 \\ -1/2 \\ 1
\end{bmatrix}
\begin{bmatrix}
-3/2 \\ -1/2 \\ 1
\end{bmatrix}
\begin{bmatrix}
-5/4 \\ -1/4 \\ 1/2
\end{bmatrix}, \quad \omega_6 = [0,0,0,0];
\qquad 
D_7: \begin{bmatrix}
0 \\ 1/3 \\ 2/3 
\end{bmatrix}
\begin{bmatrix}
1/4 \\ 1/4 \\ 1/2
\end{bmatrix}
\begin{bmatrix}
-3/4 \\ 1/4 \\ 1/2
\end{bmatrix}
\begin{bmatrix}
-1/4 \\ -1/4 \\ 1/2
\end{bmatrix}, \quad \omega_7= [0,0,0,0];
\]
\[
D_8: \begin{bmatrix}
-1/2 \\ 1/2 \\ 1
\end{bmatrix}
\begin{bmatrix}
-2/3 \\ -1/3 \\ 2/3 
\end{bmatrix}
\begin{bmatrix}
-1/4 \\ -1/4 \\ 1/2
\end{bmatrix}
\begin{bmatrix}
-5/4 \\ -1/4 \\ 1/2
\end{bmatrix}, \quad \omega_8 = [2,0,0,0];
\qquad
D_9: \begin{bmatrix}
2/3 \\ 1/2 \\ 1
\end{bmatrix}
\begin{bmatrix}
1/2 \\ -1/2 \\ 1
\end{bmatrix}
\begin{bmatrix}
1 \\ 0 \\ 2/3
\end{bmatrix}
\begin{bmatrix}
1/4 \\ 1/4 \\ 1/2
\end{bmatrix}, \quad \omega_9 = [1,-1,1,-1];
\]
\[
D_{10}: \begin{bmatrix}
1/2 \\ -1/2 \\ 1
\end{bmatrix}
\begin{bmatrix}
1 \\ 0 \\ 2/3
\end{bmatrix}
\begin{bmatrix}
1/4 \\ 1/4 \\ 1/2
\end{bmatrix}
\begin{bmatrix}
3/4 \\ -1/4 \\ 1/2
\end{bmatrix}, \quad \omega_{10} = [0,-1,2,-1];
\qquad
D_{11}:  \begin{bmatrix}
3/2 \\ 1/2 \\ 1
\end{bmatrix}
\begin{bmatrix}
-1/2 \\ 1/2 \\ 1
\end{bmatrix}
\begin{bmatrix}
1/2 \\ -1/2 \\ 1
\end{bmatrix}
\begin{bmatrix}
-1/4 \\ -1/4 \\ 1/2
\end{bmatrix}, \quad \omega_{11} =[0,0,0,0].
\]
}

\subsection{Maximal Return Domains for Symmetry 1}\label{data_sym1}

Here is the list of all the maximal return domains in $\mathbf A$ along with their coefficient tuple of matrix format:

{\tiny
\[
\alpha_0: \begin{bmatrix}
-5 & 1 & -1 & 0 \\
3 & -2 & -1 & 0 \\
-1 & 0 & 3 & -2 \\
2 & -2 & -2 & 2 
\end{bmatrix}, \quad 
\alpha_1: \begin{bmatrix}
7 & -6 & -1 & 0 \\
-5 & 2 & -1 & 0 \\
-3 & 0 & 9 & -6 \\
-9 & 4 & 3 & -2 \\
8 & -4 & 0 & 0 
\end{bmatrix}, \quad
\alpha_2: \begin{bmatrix}
-5 & 2 & -1 & 0 \\
1 & -2 & -7 & 4 \\
-11 & 6 & -7 & 4 \\
-4 & 2 & 4 & -2 
\end{bmatrix}
\]
\[
\alpha_3: \begin{bmatrix}
-3 & 0 & 9 & -6 \\
3 & -4 & 3 & -2 \\
-9 & 4 & 3 & -2 \\
8 & -4 & 0 & 0
\end{bmatrix}, \quad
\alpha_4: \begin{bmatrix}
1 & -2 & -7 & 4\\
-11 & 6 & -7 & 4 \\
11 & -4 & 3 & -2 \\
-1 & 0 & 3 & -2 \\
6 & -4 & 6 & -4
\end{bmatrix}, \quad
\alpha_5: \begin{bmatrix}
1 & -2 & -7 & 4\\
-9 & 4 & 3 & -2 \\
11 & -4 & 3 & -2\\
0 & 0 & 0 & 0
\end{bmatrix}
\]
\[
\alpha_6: \begin{bmatrix}
23 & -8 & 3 & -2 \\
11 & -8 & 3 & -2 \\
17 & -12 & -3 & 2\\
-6 & 4 & 6 & -4
\end{bmatrix},\quad
\alpha_7: \begin{bmatrix}
11 & -8 & 3 & -2 \\
-1 & 0 & 3 & -2 \\
5 & -2 & -3 & 2 \\
0 & 0 & 0 & 0
\end{bmatrix},\quad
\alpha_8: \begin{bmatrix}
3 & -4 & 3 & -2 \\
23 & -16 & 3 & -2\\
-3 & 0 & -3 & 2 \\
17 & -12 & -3 & 2\\
-12 & 8 & 0 & 0
\end{bmatrix}
\]
\[
\alpha_{9}: \begin{bmatrix}
11 & -8 & 3 & -2 \\
5 & -2 & -3 & 2 \\
7 & -6 & 7 & -4\\
-5 & 2 & 7 & -4\\
-4 & 2 & -4 & 2 
\end{bmatrix}, \quad
\alpha_{10}: \begin{bmatrix}
-3 & 0 & -3 & 2 \\
17 & -12 & -3 & 2 \\
7 & -6 & 7 & -4 \\
-2 & 2 & 2 & -2 
\end{bmatrix}, \quad
\alpha_{11}: \begin{bmatrix}
7 & -6 & 7 & -4 \\
-5 & 2 & 7 & -4 \\
1 & -2 & 1 & 0\\
-4 & 2 & -4 & 2
\end{bmatrix}
\]
}

\subsection{Maximal Return Domains for Symmetry 2}\label{data_sym2}

Here are 10 maximal return domains $\alpha_0, \cdots, \alpha_9$ in $\mathbf Q$ along with their coefficient tuples of matrix format:

{\tiny
\[
\alpha_0: \begin{bmatrix}
7 & -4 & 3 & -2 \\
-3 & 2 & 13 & -8 \\
2 & 0 & -3 & 2 \\
-9 & 6 & 7 & -4 \\
6 & -4 & -6 & 4
\end{bmatrix},\quad
\alpha_1: \begin{bmatrix}
-3 & 2 & 13 & -8 \\
13 & -8 & -3 & 2 \\
-19 & 12 & -3 & 2 \\
4 & -2 & 4 & -2
\end{bmatrix},\quad
\alpha_2: \begin{bmatrix}
13 & -8 & -3 & 2 \\
-19 & 12 & -3 & 2 \\
3 & -2 & 7 & -4 \\
-29 & 18 & 7 & -4 \\
20 & -12 & 0 & 0 
\end{bmatrix}
\]
\[
\alpha_3: \begin{bmatrix}
2 & 0 & -3 & 2 \\
-9 & 6 & 7 & -4 \\
11 & -6 & 7 & -4 \\
0 & 0 & 0 & 0
\end{bmatrix}, \quad 
\alpha_4: \begin{bmatrix}
-19 & 12 & -3 & 2 \\
-29 & 18 & 7 & -4 \\
-9 & 6 & 7 & -4 \\
10 & -6 & 10 & -6
\end{bmatrix},\quad
\alpha_5: \begin{bmatrix}
23 & -14 & 7 & -4 \\
-9 & 6 & 7 & -4 \\
7 & -4 & -9 & 6\\
0 & 0 & 0 & 0
\end{bmatrix}
\]

\[
\alpha_6: \begin{bmatrix}
-9 & 6 & 7 & -4 \\
11 & -6 & 7 & -4 \\
7 & -4 & -9 & 6 \\
27 & -16 & -9 & 6 \\
 -10 & 6 & 10 & -6
\end{bmatrix}, \quad
\alpha_7: \begin{bmatrix}
3 & -2 & 7 & -4 \\
23 & -14 & 7 & -4 \\
-3 & 2 & 1 & 0\\
17 & -10 & 1 & 0 \\
-12 & 8 & 0 & 0
\end{bmatrix},\quad
\alpha_8: \begin{bmatrix}
7 & -4 & -9 & 6 \\
27 & -16 & -9 & 6 \\
17 & -10 & 1 & 0 \\
6 & -4 & 6 & -4 
\end{bmatrix}
\alpha_{9}: \begin{bmatrix}
11 & -6 & 7 & -4 \\
17 & -10 & 1 & 0 \\
5 & -2 & 1 & 0 \\
-4 & 2 & -4 & 2 
\end{bmatrix}
\]

}

\clearpage{}
\bibliography{triplePets} 
\bibliographystyle{plain}

\end{document}